\documentclass[a4paper,twoside,10pt]{article}

\usepackage[english]{babel}
\usepackage[utf8]{inputenc}
\usepackage{lmodern,color}
\usepackage[T1]{fontenc}
\usepackage{indentfirst}
\usepackage{amsmath,amssymb}
\usepackage[percent]{overpic}
\usepackage[a4paper,top=3cm,bottom=3cm,left=3cm,right=3cm,%
bindingoffset=5mm]{geometry}
\usepackage{lmodern}
\usepackage[T1]{fontenc}
\usepackage{lettrine}
\usepackage{booktabs}
\usepackage{subcaption}
\usepackage{authblk}
\usepackage{emp}
\usepackage{amsthm}
\usepackage{amsmath,amssymb,amsthm}
\usepackage{braket}
\usepackage{chngpage}
\usepackage{cases}
\usepackage{graphicx}
\usepackage{epstopdf}
\usepackage{tabularx}
\usepackage{multirow}
\usepackage{enumerate}
\usepackage{mathtools}
\usepackage{bm}

\newcommand{\numberset}{\mathbb}
\newcommand{\N}{\numberset{N}}
\newcommand{\R}{\numberset{R}}

\newcommand{\Pk}{\numberset{P}}

\renewcommand{\epsilon}{\varepsilon}
\renewcommand{\theta}{\vartheta}
\renewcommand{\rho}{\varrho}
\renewcommand{\phi}{\varphi}

\newcommand{\xx}{\boldsymbol{x}}

\newcommand{\nn}{\boldsymbol{n}}

\newcommand{\aaa}{\boldsymbol{\alpha}}
\newcommand{\bb}{\boldsymbol{\beta}}

\newcommand{\vbf}{\mathbf{v}}
\newcommand{\Vbf}{\mathbf{V}}
\newcommand{\Zbf}{\mathbf{Z}}
\newcommand{\xpbf}{\mathbf{x}^\perp}
\newcommand{\Deltabf}{\mathbf{\Delta}}
\newcommand{\ubf}{\mathbf{u}}
\newcommand{\wbf}{\mathbf{w}}
\newcommand{\fbf}{\mathbf{f}}
\newcommand{\nablabf}{\bm{\nabla}}
\newcommand{\ebf}{\bm{\varepsilon}}

\DeclareMathOperator{\Divbf}{\mathbf{div}}
\DeclareMathOperator{\Div}{div}

\def\P0{{\Pi^{0, E}_{k}}}
\def\Pg{{\Pi^{0}_{k}}}
\def\PZ0P{{\boldsymbol{\Pi}^{0, E}_{k-1}}}
\def\PP0P{{\boldsymbol{\Pi}^{0, E}_{k}}}

\def\cskew{{c^{\rm skew}}}

\def\cskewE{{c^{{\rm skew},E}}}
\def\cskewEh{{c^{{\rm skew},E}_h}}

\def\cskewh{{c^{{\rm skew}}_h}}

\newcommand{\errbf}{\mathbf{e}}

{\left\lbrace\begin{array}{@{}l@{}}}%
{\end{array}\right.}
\theoremstyle{definition}

\theoremstyle{remark}
\newtheorem{remark}{Remark}[section]

\theoremstyle{remark}

\theoremstyle{plain}
\newtheorem{theorem}{Theorem}[section]
\newtheorem{proposition}{Proposition}[section]
\newtheorem{corollary}{Corollary}[section]
\newtheorem{lemma}{Lemma}[section]

\usepackage{lipsum}

\usepackage{authblk}
\usepackage{color}

\usepackage{setspace}

\author[1]{M. Trezzi 
\thanks{manuel.trezzi@unimib.it}}

\affil[1]{Dipartimento di Matematica e Applicazioni, Universit\`a degli Studi di Milano Bicocca, Via Roberto Cozzi 55 - 20125 Milano, Italy}  

\title{A three-level CIP-VEM approach for the Oseen equation}

\begin{document}
\maketitle

\begin{abstract}
We study a pressure-robust virtual element method for the Oseen problem. In the advection-dominated case, the method is stabilized with a three level jump of the convective term.
To analyze the method, we prove specific estimates for the virtual space of potentials.
Finally, e prove stability of the proposed method in the advection-dominated limit and derive $h$-version error estimates for the velocity and the pressure.
\end{abstract}

\section{Introduction}

In recent years, there has been significant interest in developing pressure-robust numerical schemes \cite{JLMNR:2017, CQD:2023, BLV:2017}. These methods enable accurate approximations of velocity, particularly when dealing with non-smooth pressure, by eliminating the dependence on pressure in the error analysis of the velocity.
This advancement is crucial for ensuring the reliability and robustness of simulations in various fluid dynamics applications.
This work aims to introduce a pressure-robust Virtual Element Method (VEM) for the Oseen problem (the linearized version of the Navier-Stokes equations), which remains stable even in advection-dominated regimes. It is well known that when the diffusive coefficient becomes small with respect to the advection parameter, standard numerical schemes produce unsatisfactory solutions when the Péclet number is not sufficiently small.\medskip

The VEM, introduced in 2013 \cite{volley}, represents an advanced evolution of the classical Finite Element Method (FEM) to solve Partial Differential Equations (PDEs). One of its most attractive features is its ability to manage complex polygonal/polyhedral meshes, allowing each element to adopt arbitrary shapes. 
This flexibility is particularly advantageous in addressing mechanical problems involving intricate geometries; see for instance \cite{sema-simai} and \cite{ACTA-VEM} and the references therein. 
The VEM has found applications in numerous fields, including structural mechanics, fluid dynamics, and geophysics, owing to its ability to handle complex domain geometries and its robustness in numerical simulations.
The term \emph{virtual} in VEM highlights that it does not necessitate explicit knowledge of the basis functions; instead, only certain information is needed to construct the stiffness matrix.\medskip

In \cite{BLV:2017, BDDD:2022}, the authors propose a VEM for the Stokes equations that achieves divergence-free conditions by ensuring that the divergence of a virtual velocity is included in the space of the pressures in the definition of the local spaces. 
However, this requirement does not entirely eliminate the dependence on pressure in the error analysis of the velocity. 
A slight dependence on the pressure still exists due to the approximation of the right-hand side. 
As a result, the discrete scheme does not fully remove the pressure dependence but significantly reduces it.
There exists a VEM for the Oseen problem that remains stable in the hyperbolic limit, as introduced in \cite{BDV:2021}. 
To achieve stabilization, this method incorporates local SUPG-like terms for the vorticity equation and a jump term.
\medskip

A first FEM with CIP stabilization for the Oseen equation was presented in \cite{oseenFEM}. 
However, this method has the disadvantage of not being pressure-robust, meaning that the error estimate for the velocity depends on the pressure.
To address this issue, a pressure-robust FEM with a three-level CIP stabilization was introduced in \cite{BBCG:2024}.


Following the FEM method \cite{BBCG:2024}, we try to develop a VEM that achieves stability solely through jump operators applied to the skeleton of the mesh and conclude the analysis presented in \cite{trezzi2024}.
The method controls the polynomial parts of the jumps of $(\nablabf \ubf)\bb$ through a three-level CIP-form.
Specifically, we control
the jumps of $(\nablabf \ubf) \bb$, the jump of the $\text{curl}$ of $(\nablabf \ubf) \bb$, and the jump of the gradient of the $\text{curl}$ of $(\nablabf \ubf) \bb$.
Here, we are considering the scalar $\text{curl}$ applied to vector-valued functions defined as
\[
\text{curl} (\vbf) \coloneqq \frac{\partial v_2}{\partial x} - \frac{\partial v_1}{\partial y} \, .
\]
where $v_1$ and $v_2$ denotes the two components of the vector-valued function $\vbf$.We mention that, as is typical in a VEM context, we control a polynomial projection of these quantities, while the remaining part is controlled by a stabilization term.\medskip

To perform the theoretical analysis of the method, we introduce an Oswald interpolant \cite{feI} into the space of the stream function, similar to the one presented in \cite{BLT:2024,LPT:2024} in a VEM context. This operator maps piecewise smooth functions into the virtual element space. We emphasize that this operator is not defined on the velocity space but rather on the space of potentials.
The degrees of freedom in this space include not only the pointwise values of the function at certain points on the boundary but also the values of the gradient.
This implies that the difference between a function and its Oswald interpolant is controlled not only by the jumps of the function but also by the jumps of its gradient.
Furthermore, we had to prove certain inverse estimates for the space of potentials which, to the best of our knowledge, have never been proved before. 
In particular, we have to prove specific inverse estimates for certain Sobolev norms and a sort of inverse trace inequality for a family of virtual functions.
\medskip

This paper is organized as follows: the second section introduces the Oseen problem and the spaces used in the analysis.
The third section describes the method. The fourth section presents the theoretical analysis. 
Finally, we conclude with some numerical results in the fifth section.

\section{Model problem} \label{s:oseen-model-problem}

Given a polygonal and simple connected domain $\Omega \subset \R^2$ with boundary $\Gamma$, we consider the steady Oseen equation with homogenous Dirichlet boundary conditions:
\begin{equation} \label{eq:oseen-strong}
\left \{
\begin{aligned}
&\text{find $(\ubf,p)$ such that:}& \\
&- \nu \, \Divbf(\ebf(\ubf)) + (\nablabf \ubf) \bb + \sigma \, \ubf - \nabla p = \fbf \, \quad &\text{in }\Omega \, , \\
&\Div(\ubf) = 0 \, \quad &\text{in }\Omega \, , \\
&\ubf = 0 \, \quad &\text{on }\Gamma \, .
\end{aligned}
\right .     
\end{equation}
As usual in this type of problems, $\ubf$ denotes the velocity of the fluid while $p$ is the pressure. 
Furthemore, $\Divbf$ ($\Div$), $\nablabf$, $\nabla$, denote the vector (scalar) divergence operator, the gradient operator for vector fields and the gradient operator for scalar function while $\ebf(\ubf)$ is the symmetric gradient operator defined as
\begin{equation*} \label{eq:oseen-symm_grad}
\ebf(\ubf) 
\coloneqq
\dfrac{\nablabf \ubf + \nablabf^T \ubf}{2} \, .
\end{equation*}
The parameters $\nu, \sigma \in \R^+$ represent the diffusive and reaction coefficients, respectively. 
The transport advective field $\bb \in [W^1_\infty(\Omega)]^2$ is a sufficiently smooth vector-valued function that satisfies $\Div (\bb) = 0$.
Finally, the load term $\fbf$ is assumed to be $\fbf \in [L^2(\Omega)]^2$.
We introduce the spaces
\begin{equation*} 
\Vbf(\Omega) \coloneqq [H_0^1(\Omega)]^2 \quad \text{and} \quad
Q(\Omega) \coloneqq L^2_0(\Omega) =\{ v \in L^2(\Omega) \text{ s.t. }(v,1) = 0 \} \, .
\end{equation*}
A possible variational formulation for problem \eqref{eq:oseen-strong} reads as
\begin{equation} \label{eq:oseen-variational-equation}
\left \{
\begin{aligned}
&\text{find $(\ubf,p) \in \Vbf(\Omega) \times Q(\Omega)$ such that:}& \\
&A(\ubf, \vbf) + c(\ubf, \vbf) + b(\vbf, p) = \mathcal F(\vbf) \quad &\forall \vbf \in \Vbf(\Omega) \, , \\
&b(\ubf, q) = 0 \quad &\forall q \in Q(\Omega) \, ,
\end{aligned}
\right .
\end{equation}
where the bilinear forms are defined as
\begin{equation*}
A(\cdot,\cdot): \Vbf(\Omega) \times \Vbf(\Omega) \to \R \, , \quad A(\ubf, \vbf) \coloneqq \nu \int_\Omega \ebf (\ubf) : \ebf(\vbf) \, {\rm d} \Omega 
+
\sigma \int_\Omega \ubf \cdot \vbf \, {\rm d} \Omega \, ,
\end{equation*}
\begin{equation*}\label{eq:oseen-b_bf}
b(\cdot,\cdot): \Vbf(\Omega) \times Q(\Omega) \to \R \, , \quad b(\vbf, q) \coloneqq \int_\Omega q \, \Div(\vbf) \, {\rm d} \Omega \, ,
\end{equation*}
\begin{equation*}\label{eq:oseen-c_bf}
c(\cdot,\cdot): \Vbf(\Omega) \times \Vbf(\Omega) \to \R \, , \quad c(\ubf, \vbf) \coloneqq \int_\Omega [(\nablabf u) \bb] \cdot \vbf \,  {\rm d} \Omega \, ,
\end{equation*}
and as usual $\mathcal F : \Vbf(\Omega) \to \R$ is the $[L^2(\Omega)]^2-$inner product against the function $\fbf$
\[
\mathcal F(\vbf) \coloneqq \int_\Omega \fbf \cdot \vbf \, {\rm d}\Omega \, .
\]
We also define the bilinear form
\begin{equation}\label{eq:k}
\mathcal K(\cdot,\cdot): \Vbf(\Omega) \times \Vbf(\Omega) \to \R \, , \quad \mathcal K(\ubf, \vbf) \coloneqq A(\ubf,\vbf)  
+
c(\ubf,\vbf) \, .
\end{equation}
Thanks to the assumption $\Div(\bb) = 0$ and integration by parts, we mention that the bilinear form $c(\cdot,\cdot)$ is skew-symmetric.
Hence, it is equivalent to its skew-symmetric part defined as
\[
c^{\text{skew}}(\ubf,\vbf)
\coloneqq
\frac{1}{2} 
\bigl (c(\ubf,\vbf) - c(\vbf, \ubf)  \bigr) = c(\ubf,\vbf) \, \quad \forall \ubf,\vbf \in \Vbf(\Omega).
\]
Introducing the space
\begin{equation*}
\Zbf(\Omega) \coloneqq \Big \{ \vbf \in \Vbf(\Omega) \text{ such that } \Div (\vbf) = 0 \Big \} \, ,
\end{equation*}
we can reformulate problem \eqref{eq:oseen-strong} in a pressure independent form
\begin{equation} \label{eq:oseen-kernel-form}
\left \{
\begin{aligned}
&\text{find $\ubf \in \Zbf(\Omega)$ such that: }\\
&A(\ubf, \vbf) + c(\ubf, \vbf) = (\fbf, \vbf) \, \quad \forall \vbf \in \Zbf(\Omega) \, .
\end{aligned}
\right .    
\end{equation}
This new formulation is useful for the theoretical analysis but is impractical for implementation due to the difficulty in identifying divergence-free functions.
Since $\Omega$ is a smooth and simply connected domain with Lipschitz boundary, we can associate a potential $\varphi$ to a divergence free function $\vbf \in \Zbf(\Omega)$ such that
\[
\textbf{curl} (\varphi) 
\coloneqq 
\left ( \frac{\partial \varphi}{\partial y} , - \frac{\partial \varphi}{\partial x} \right)^T
= 
\vbf \, .
\]
Hence, we introduce a useful space for the analysis of the method. 
We define the space of the stream functions as
\begin{equation*}
\Phi(\Omega)
\coloneqq
\{
\varphi \in H^2(\Omega)
\text{ such that }
\varphi|_\Gamma = \nabla \varphi|_\Gamma \cdot \nn  = 0
\} \, .
\end{equation*}
Thanks to this space, the following sequence is exact on simple connected domains
\begin{equation}\label{eq:oseen-ex}
0
\xrightarrow{\mathmakebox[3em]{i}}
\Phi(\Omega)
\xrightarrow{\mathmakebox[3em]{\textbf{curl}}}
\Vbf(\Omega)
\xrightarrow{\mathmakebox[3em]\Div}
Q(\Omega)
\xrightarrow{\mathmakebox[3em]{0}}
0 \, ,
\end{equation}
where $i$ is the operator that associates to each real number the corresponding constant function.
The term \textit{exact} means that the range of each operator is equal to the kernel of the following operator. 
In particular the following equivalence holds
\[
\textbf{curl}(\Phi(\Omega)) = \Zbf(\Omega) \, .
\]
In the hyperbolic limit, when the parameter 
$\bb$ dominates over the others, standard discretizations of \eqref{eq:oseen-variational-equation} require stabilization. Without such stabilization, the numerical solutions obtained are unsatisfactory, exhibiting non-physical oscillations across element boundaries.
In the following sections, we describe how devise a VEM that is stable in the hyperbolic limit of  Oseen equation.

\section {The method}

\subsection{Mesh assumptions}
In this section, we introduce the notations and assumptions related to the decompositions of the domain $\Omega$ that will be considered.
We consider a sequence $\set{ \Omega_h }_h$ of decompositions of the domain $\Omega$ composed of non-overlapping (open) polygons $E \in \Omega_h$.
Here, $h$ denotes the maximum of the diameters of the elements in $\Omega_h$
\[
h \coloneqq \max_{E \in \Omega_h} h_E \, ,
\]
where $h_E$ is the diameter of the element $E \in \Omega_h$.
We suppose that $\set{\Omega_h}_h$ satisfies the following assumption:\\
\textbf{(A1) Mesh assumption.}
There exists a positive constant $\rho$ such that for any $E \in \set{\Omega_h}_h$:
\begin{itemize}
\item $E$ is star-shaped with respect to a ball $B_E$ of radius grater or equal than $  \rho \, h_E$,
\item any edge $e$ of $E$ has length grater or equal than $\rho \, h_E$.
\end{itemize} 
From now on, we denote by $|E|$ and $\nn^E$ the area and the unit outward normal of the polygon $E$, respectively.
The restriction of the unit normal to an edge $e \subset \partial E$ is denoted with $\nn^e$, while the tangent unit vector is denoted with $\mathbf t^e$.
The set of edges of a tessellation $\Omega_h$ is denoted by $\mathcal{E}_h$.
This set is divided into internal edges and boundary edges 
\[
\mathcal{E}_h = \mathcal{E}_h^o \cup \mathcal{E}_h^\partial \, .
\]
Given an interior edge $e \subset \partial E \cap \partial K$ for $E,K \in\Omega_h$, we define for a sufficiently smooth function $w$ the jump operator as
\[
[\![ w ]\!]_e = \lim_{s \to 0} w(\mathbf{x} - s\nn^E) + w(\mathbf{x} - s\nn^K) \, .
\]
If $e$ is a boundary edge, we set $ [\![ w ]\!]_e = w$.
We omit the subscript $e$ when the edge under consideration is clear from the context.
Let's introduce some basic spaces that will be useful later on.
Given two positive integers $n$ and $m$, and $p\in [0,+\infty]$, for any $E \in \Omega_h$ we define
\begin{itemize}
\item $\Pk_n(E)$: the set of polynomials on $E$ of degree lesser or equal than $n$  (with $\Pk_{-1}(E)=\{ 0 \}$),
\item $\Pk_n(\Omega_h) := \{q \in L^2(\Omega) \quad \text{such that} \quad q|_E \in  \Pk_n(E) \quad \text{for all $E \in \Omega_h$}\}$,
\item $W^m_p(\Omega_h) := \{v \in L^2(\Omega) \text{ such that } v|_E \in  W^m_p(E) \text{ for all $E \in \Omega_h$}\}$ equipped with the broken norm and seminorm
\[
\begin{aligned}
&\|v\|^p_{W^m_p(\Omega_h)} := \sum_{E \in \Omega_h} \|v\|^p_{W^m_p(E)}\,,
\qquad 
&|v|^p_{W^m_p(\Omega_h)} := \sum_{E \in \Omega_h} |v|^p_{W^m_p(E)}\,,
\qquad & \text{if $1 \leq p < \infty\, ,$}
\\
&\|v\|_{W^m_p(\Omega_h)} := \max_{E \in \Omega_h} \|v\|_{W^m_p(E)}\,,
\qquad 
&|v|_{W^m_p(\Omega_h)} := \max_{E \in \Omega_h} |v|_{W^m_p(E)}\,,
\qquad & \text{if $p = \infty\, ,$}
\end{aligned}
\]
In the case $p=2$, we set
\[
\begin{aligned}
&\|v\|^2_{m,\Omega_h} := \| v\|^2_{W^m_p(\Omega_h)}\,,
\qquad 
&|v|^2_{m,\Omega_h} := | v|^2_{W^m_p(\Omega_h)}\,.
\end{aligned}
\]
\end{itemize}
If $\mathcal D$ denotes one of the spaces introduced above, the notation $[\mathcal D]^l$ refers to the extension of $\mathcal D$ to vector-valued functions of dimension $l$.
Given an element $E \in \Omega_h$ We also introduce the following polynomial projections, which are essential in any VEM implementation:
\begin{itemize}
\item the $\boldsymbol{L^2}$\textbf{-projection} $\Pi_n^{0, E} \colon [L^2(E)]^2 \to [\Pk_n(E)]^2$, given by
\begin{equation*}
\label{eq:P0_k^E}
\int_E \mathbf{q}_n \cdot (\vbf - \, {\Pi}_{n}^{0, E}  \vbf) \, {\rm d} E = 0 \qquad  \text{for all $\vbf \in [L^2(E)]^2$  and $\mathbf{q}_n \in [\Pk_n(E)]^2 \, .$} 
\end{equation*} 
The extension of this operator to $2\times2$ tensors is denoted with the bold symbol $\mathbf{\Pi}_n^{0,E}: [L^2(E)]^{2\times 2} \to [\Pk_n(E)]^{2\times 2}$,
\item the $\boldsymbol{H^1}$\textbf{-seminorm projection} ${\Pi}_{n}^{\nabla,E} \colon [H^1(E)]^2 \to [\Pk_n(E)]^2$, defined by 
\begin{equation*}
\label{eq:Pn_k^E}
\left\{
\begin{aligned}
& \int_E \nabla  \, \mathbf{q}_n \cdot \nabla ( \vbf - \, {\Pi}_{n}^{\nabla,E}   \vbf)\, {\rm d} E = 0 \quad  \text{for all $\vbf \in H^1(E)$ and  $\mathbf{q}_n \in \Pk_n(E)$,} \\
& P_0(\vbf - \Pi^{\nabla,E}_n \vbf) = 0 \, ,
\end{aligned}
\right.
\end{equation*}
here $P_0:[H^1(E)]^2 \to \R^2$ is any projection operator onto constants. 
\end{itemize}
We denote with $\Pi^0_k$ the global version of $\P0$ obtained by gluing together the local projections
In the following, we will frequently use the following lemmas \cite{BDV:2021, brenner-scott:book}.
\begin{lemma}\label{lm:stab}
Let $\vbf \in \Zbf(\Omega) \cap [H^s(\Omega_h)]^2$ be a smooth function with $s \geq 1$.
There exists a potential $\psi \in \Phi(\Omega) \cap H^{s+1}(\Omega_h)$ such that $\textnormal{\textbf{curl}}(\psi) = \vbf$ and
\[
| \psi |_{H^{s+1}(\Omega_h)}
\lesssim
| \vbf |_{[H^s(\Omega_h)]^2}\, .
\]
\end{lemma}

\begin{lemma}\label{lm:bh}
Under assumption \textbf{(A1)}, for any $E \in \Omega_h$ and for any sufficiently smooth function $\phi \in H^s(E)$ defined on $E$, we have that 
\[
\begin{aligned}
&\|\phi - \Pi^{0,E}_n \phi\|_{W^m_p(E)} \lesssim h_E^{s-m} |\phi|_{W^s_p(E)} 
\qquad & \text{$s,m \in \N$, $m \leq s \leq n+1$, $p=1, \dots, \infty \, ,$}
\\
&\|\phi - \Pi^{\nabla,E}_n \phi\|_{m,E} \lesssim h_E^{s-m} |\phi|_{s,E} 
\qquad & \text{$s,m \in \N$, $m \leq s \leq n+1$, $s \geq 1\, ,$}
\\
&\|\nabla \phi - \boldsymbol{\Pi}^{0,E}_{n} \nabla \phi\|_{m,E} \lesssim h_E^{s-1-m} |\phi|_{s,E} 
\qquad & \text{$s,m \in \N$, $m+1 \leq s \leq n+2\, ,$}
\end{aligned}
\]
with obvious extension to vector valued functions.
\end{lemma}

\subsection{Virtual element spaces} \label{s: local_space}
Given an integer $k \geq 2$ and an element $E \in \Omega_h$, we introduce the enhanced virtual element \cite{BMV:2019} space as
\begin{equation} \label{eq:oseen-local-space}
\begin{aligned}
\Vbf_h^k(E) \coloneqq \Big \{ \vbf_h \in [H^1(E)]^2 \text{ s.t. }  &(i) \, \Deltabf \vbf_h + \nabla s \in \xpbf\mathbb{P}_{k-1}(E), \quad \text{for some } s \in L_0^2(E) \, , \\
&(ii)  \, {\rm div}(\vbf_h) \in \mathbb{P}_{k-1}(E) \, , \\
&(iii)  \, \vbf_h|_e \in [\mathbb P_k(e)]^2 \, , \quad \forall e \in \partial E \, \text{ and }v_h|_{\partial E} \in [C^0(\partial E)]^2 , \\
&(iv) \, (\vbf_h - \Pi^{\nabla,E}_k \vbf_h, \xpbf \hat p_{k-1})_{0,E}=0 \, ,  \quad \forall \hat p_{k-1} \in \hat{\mathbb P}_{k-1 / k-3}(E)  \Big\} \, .
\end{aligned}
\end{equation}
where the vector $\xpbf$ is defined as $\xpbf \coloneqq [x_2; -x_1]^T$.
we consider the following set of linear operator as set of DoFs for the space $\Vbf_h^k(E)$:
\begin{itemize}
    \item the pointwise values of $\vbf_h$ at the vertices of the polygon $E$,
    \item the pointwise values of $\vbf_h$ at $k-1$ internal points of a Gauss-Lobatto quadrature for every edge $e \subset \partial E$,
    \item the moments of $\vbf_h$
    \[
    \frac{1}{|E|} \int_E \vbf_h \cdot \mathbf m^\perp \left( \dfrac{\xx - \xx_E}{h_E} \right)^{\aaa} {\rm d}E  \quad |\aaa| \leq k-3 \, , 
    \]
    where $\mathbf{m}^\perp \coloneqq \frac{1}{h_E}(x_2 - x_{2,E}, -x_1 + x_{1,E})$,
    \item the moments of $\Div \vbf_h$
    \[
    \frac{h_E}{|E|} \int_E \Div (\vbf_h) \, \left( \dfrac{\xx - \xx_E}{h_E} \right)^{\aaa} \, {\rm d}E \quad 0 < | \aaa | \leq k-1 \, .
    \]
\end{itemize} 
The global space for the velocities is obtained as in the scalar case by gluing togheter the local spaces
\[
\Vbf_h^k(\Omega_h)
\coloneqq
\Big\{
\vbf_h \in \Vbf(\Omega) \text{ such that } \vbf_h|_E \in \Vbf^k_h(E) \quad \forall E \in \Omega_h
\Big\} \, .
\]
The space of pressures, is discretized by the standard piecewise polynomials space
\begin{equation*} 
Q^k_h (\Omega_h)
\coloneqq
\Big \{
q_h \in L^2_0(\Omega) \text{ such that } q_h|_E \in \mathbb{P}_{k-1}(E) \quad \forall E \in \Omega_h
\Big \} \, .
\end{equation*}
In \cite{BLV:2017}, it was proved that the couple $[\Vbf_h^k(\Omega_h), \, Q_h^k(\Omega_h)]$ is inf-sup stable for $k \geq 2$.
It holds that
\begin{equation*}\label{eq:oseen-infsup}
\sup_{\vbf_h \in \Vbf^k_h(\Omega_h)} \frac{b(\vbf_h, q_h)}{\| \nablabf \vbf_h \|_{0,\Omega_h}} 
\geq
\hat \kappa
\| q_h \|_{0,\Omega_h}
\quad
\forall q_h \in Q_h^k(\Omega_h) \, ,
\end{equation*}
where $\hat \kappa$ denotes the inf-sup constant that does not depend on the mesh size $h$.
We now introduce the discrete kernel as
\[
\Zbf_h(\Omega_h) \coloneqq \{ \vbf_h \in \Vbf_h^k(\Omega_h) \text{ such that }   b(\vbf_h, q_h) = 0 \quad \forall q_h \in Q_h^k(\Omega_h) \}.
\]
Thanks to property $(ii)$ in \eqref{eq:oseen-local-space}, the following inclusion holds
\[
\Zbf_h(\Omega_h) \subseteq \Zbf(\Omega) \, .
\]
This means that the virtual functions in the discrete kernel are divergence-free. 
Now we construct the space of the discrete stream-functions \cite{BMV:2019}. 
We define
\begin{equation*} 
\begin{aligned}
\Phi^k_h(E) \coloneqq \Big \{ \varphi \in [H^2(\bar E)]^2 &\text{ such that }  (i) \, \Delta^2 \varphi \in \mathbb{P}_{k-1}(E) \, , \\
&(ii)  \, \varphi|_e \in \mathbb P_{k+1}(e) \, , \quad \forall e \in \partial E \, , \text{ and }v_h|_{\partial E} \in C^0(\partial E) \, , \\
&(iii)  \, \nabla \varphi|_e \in \mathbb [\mathbb P_{k}(e)]^2 \, , \quad \forall e \in \partial E \, , \text{ and }\nabla v_h|_{\partial E} \in [C^0(\partial E)]^2 \, , \\
&(iv) \, (\textbf{curl}(\varphi) - \Pi^{\nabla,E}_k \textbf{curl}(\varphi), \xpbf \hat p_{k-1})_{0,E}=0   \quad \forall \hat p_{k-1} \in \hat{\mathbb P}_{k-1 / k-3}(E)  \Big\} \, ,
\end{aligned}
\end{equation*}
and the corresponding global space
\[
\Phi_h^k(\Omega_h) \coloneqq \{ \varphi \in \Phi(\Omega) \text{ such that } \varphi|_E \in \Phi_h^k(E) \quad \forall E \in \Omega_h\} \, .
\]
The following set of linear operators is a set of DoFs for the space $\Phi_h^k(E)$:
\begin{itemize}
    \item the pointwise values of $\phi$ at the vertices of the polygon $E$,
    \item the pointwise values of $\nabla \phi$ at the vertices of the polygon $E$,
    \item the pointwise values of $\phi$ at $k-2$ distinct points of every edge $e \in \partial E$,
    \item the pointwise values of $\frac{\partial \phi}{\partial n}$ at $k-1$ distinct points of every edge $e \in \partial E$,
    \item the moments of $\textbf{curl} (\phi)$
    \[
    \int_E \textbf{curl}( \phi) \cdot \mathbf{m}^\perp p_{k-3} \, dE \quad \text{for all } p_{k-3} \in \mathbb{P}_{k-3}(E).
    \]
\end{itemize}
The first four types of DoFs are associated with the boundary of the element and are shared between elements with common edges or vertices, whereas the last type of DoFs is associated with a single element only.
We have constructed an exact subcomplex of \eqref{eq:oseen-ex}
\[
0
\xrightarrow{\mathmakebox[3em]{i}}
\Phi_h^k(\Omega_h)
\xrightarrow{\mathmakebox[3em]{\textbf{curl}}}
\Vbf_h^k(\Omega_h)
\xrightarrow{\mathmakebox[3em]\Div}
Q_h^k(\Omega_h)
\xrightarrow{\mathmakebox[3em]{0}}
0 \, .
\]
Finally, we recall from \cite{BDV:2021} two approximation results for the space $\Vbf_h^k(\Omega_h)$ and the space $\Phi_h^k(\Omega_h)$ respectively.
\begin{lemma}[Approximation with divergence-free virtual element functions]
\label{lm:oseen-interpolation}
Under assumption \textbf{(A1)}, for any $\vbf \in \Vbf(\Omega) \cap [H^{s+1}(\Omega_h)]^2$, there exists $\tilde\vbf_{\mathcal I} \in \Vbf^{k}_h(\Omega_h)$, such that for all $E \in \Omega_h$, 
\[
\|\vbf - \tilde\vbf_{\mathcal I}\|_{0,E} + h_E \|\nablabf (\vbf - \tilde\vbf_{\mathcal I})\|_{0,E} \lesssim h_E^{s+1} |v|_{s+1,E} \, , 
\]
where $0 < s \le k$.
\end{lemma}
\begin{lemma}[Approximation property of $\Phi_h^k(\Omega_h)$]\label{lm:oseen-interpolation2}
Under assumption \textbf{(A1)}, for any $\phi \in \Phi(\Omega) \cap H^{s+2}(\Omega_h)$ there exists $\tilde\varphi_{\mathcal I} \in \Phi_h^k(\Omega_h)$ such that for all $E \in \Omega_h$ it holds
\[
\| \varphi - \tilde\varphi_{\mathcal I} \|_{0,E}
+
h_E | \varphi - \tilde\varphi_{\mathcal I} |_{1,E}
+
h_E^2 | \varphi - \tilde\varphi_{\mathcal I} |_{0,E}
\lesssim h_E^{s+2} | \varphi |_{s+2,E} \, ,
\]
where $0 < s \leq k$.
\end{lemma}

\subsection{Virtual element forms and the discrete problem}

This section aims to construct a computable counterpart of the forms introduced in Section \ref{s:oseen-model-problem}.
We recall that the forms introduced in Section \ref{s:oseen-model-problem} can be decomposed into local contributions by restricting the integral to an element $E \in \Omega_h$
\[
\begin{aligned}
&A(\ubf, \vbf) \eqqcolon \sum_{E \in \Omega_h} A^E(\ubf, \vbf)\,,
\qquad 
&\cskew(\ubf, \vbf) \eqqcolon \sum_{E \in \Omega_h} \cskewE(\ubf, \vbf) \,,
\\
&b(\vbf, q) \eqqcolon \sum_{E \in \Omega_h} b^E(\vbf, q)\,,
\qquad
&\mathcal F(\vbf) \eqqcolon \sum_{E \in \Omega_h} \mathcal F^E(\vbf)\, .
\end{aligned}
\]
The bilinear form $A^E(\cdot,\cdot)$ is discretized by $A^E_h(\cdot,\cdot): \Vbf_h^k(E) \times \Vbf_h^k(E) \to \R$ defined as
\begin{equation*}
\begin{aligned}
A^E_h( \ubf_h, \vbf_h) 
&\coloneqq
\nu \int_E \PZ0P \ebf(\ubf_h) : \PZ0P  \ebf(\vbf_h) \, {\rm d}E
+ 
\sigma \int_E \P0 \ubf_h \cdot \P0  \vbf_h \, {\rm d}E \\
& \quad +
\left( \nu + \sigma \, h_E^2 \right) S^E_h \bigl( (I - \P0 )\ubf_h, (I-\P0) \vbf_h)\bigr)\, ,
\end{aligned}
\end{equation*}
where $S^E_h(\cdot, \cdot): \Vbf_h^k(E) \times \Vbf_h^k(E) \to \R$ is a VEM stabilization term that satisfies
\[
\alpha_* \| \nablabf \vbf_h \|_{0,E}
\leq
S^E_h (\vbf_h, \vbf_h)
\leq
\alpha^* \| \nablabf \vbf_h \|_{0,E} \quad
\forall \vbf_h \in \Vbf_h^k(E) \cap \text{ker}(\P0)\, ,
\]
where $\alpha_*$ and $\alpha_*$ are two uniform constants.
The convective term is replaced by $c^E_h(\cdot,\cdot): \Vbf_h^k(E) \times \Vbf_h^k(E) \to \R$ defined as
\begin{equation*}
c^E_h (\ubf_h, \vbf_h)
\coloneqq
\int_E \left [ \left( \mathbf{\Pi}_k^{0,E}\nablabf \ubf_h \right) \bb \right] \cdot \P0 \vbf_h \, {\rm d}E \, .
\end{equation*}
We consider the skew-part of this bilinear form
\begin{equation*}
\cskewEh(\ubf_h, \vbf_h) \coloneqq \frac{1}{2} \bigl(c^E_h(\ubf_h, \vbf_h) - c^E_h(\vbf_h, \ubf_h)\bigr)
\end{equation*}
In order to stabilize the method in the convection-dominated case, we introduce in the formulation of the method a three level CIP term.
First, we introduce the operator
\[
( \mathcal B \vbf ) |_E 
\coloneqq
 \text{curl} \bigl(\left( \nablabf \vbf \right) \bb \bigr) \quad \forall E \in \Omega_h \, ,
\]
defined as
\[
J^E_h(\ubf_h , \vbf_h) \coloneqq \sum_{i=1}^3 \delta_i \, J^{E,i}_h(\ubf_h , \vbf_h)
+
\delta \, h_E \, S^E_h \bigl( (I - \P0 )\ubf_h, (I-\P0) \vbf_h)\bigr) \, ,
\]
where $\delta = \max\{\delta_1, \delta_2, \delta_3 \}$, and
\begin{equation*}
\begin{aligned}
J_h^{E,1} (\ubf_h , \vbf_h)& \coloneqq \frac{1}{2} \int_{\partial E/\Gamma} h_E^2 \bigl[\!\!\bigl[\bigl( \nablabf \Pg \ubf_h \bigr) \bb \cdot \mathbf t^E \bigr ]\!\! \bigr] \bigl[\!\!\bigl[\bigl( \nablabf \Pg \vbf_h \bigr) \bb \cdot \mathbf t^E \bigr ]\!\! \bigr] \, {\rm d }s \, , \\
J_h^{E,2} (\ubf_h , \vbf_h)& \coloneqq \frac{1}{2} \int_{\partial E/\Gamma} h_E^4 \bigl[\!\!\bigl[ \mathcal B \mathbf{\Pi}_k^{0} \ubf_h\bigr ]\!\! \bigr] \bigl[\!\!\bigl[\mathcal B \mathbf{\Pi}_k^{0}\vbf_h \bigr ]\!\! \bigr] \, {\rm d }s \, , \\
J_h^{E,3} (\ubf_h , \vbf_h)& \coloneqq \frac{1}{2} \int_{\partial E/\Gamma} h_E^6 \bigl[\!\!\bigl[ \nabla \mathcal B \mathbf{\Pi}_k^{0} \ubf_h\bigr ]\!\! \bigr] \bigl[\!\!\bigl[\nabla \mathcal B \mathbf{\Pi}_k^{0} \vbf_h \bigr ]\!\! \bigr] \, {\rm d }s \, .
\end{aligned}
\end{equation*}
Thanks to the DoFs of $\Vbf_h^k(\Omega_h),$ it is not necessary to introduce a discretization of $b(\cdot, \cdot)$.
Now we introduce the local bilinear form 
\begin{equation}\label{eq:kh}
    \mathcal K_h^E(\ubf_h, \vbf_h) 
\coloneqq 
A_h^E(\ubf_h, \vbf_h)
+
\cskewEh(\ubf_h, \vbf_h)
+
J_h^E(\ubf_h, \vbf_h) \, .
\end{equation}
The right-hand side is discretized by
\[
\mathcal F_h^E(\vbf_h) \coloneqq \int_E \fbf \cdot \P0 \vbf_h \, {\rm d} E \, .
\]
We introduce the global versions of the bilenar forms by summing over all the polygons $E \in \Omega_h$
\[
\begin{aligned}
&A_h(\ubf_h, \vbf_h) := \sum_{E \in \Omega_h} A_h^E(\ubf_h, \vbf_h)\,,
\qquad 
&\cskewh(\ubf_h, \vbf_h) := \sum_{E \in \Omega_h} \cskewEh(\ubf_h, \vbf_h) \,,
\\
&J_h(\ubf_h, \vbf_h) := \sum_{E \in \Omega_h} J_h^E(\ubf_h, \vbf_h)\,,
\qquad
&\mathcal{F}_h(\vbf_h) := \sum_{E \in \Omega_h} \mathcal{F}_h^E(\vbf_h) \, ,
\end{aligned}
\]
and
\[
\mathcal K_h(\ubf_h, \vbf_h) := \sum_{E \in \Omega_h} \mathcal K_h^E(\ubf_h, \vbf_h) \, .
\]
The discrete problem reads as
\begin{equation}
\label{eq:oseen-discrete-problem}
\left \{
\begin{aligned}
& \text{find $(\ubf_h, p_h) \in [\Vbf^k_h(\Omega_h), Q_h^k(\Omega_h)]$ such that:} 
\\
& \mathcal K_h(\ubf_h, \, \vbf_h) + b(\vbf_h, p_h )= \mathcal{F}_h(\vbf_h) \quad &\text{ $\forall \vbf_h \in \Vbf^k_h(\Omega_h)$,}\\
&b(\ubf_h, q_h) = 0 \quad &\text{ $\forall q_h \in Q^k_h(\Omega_h)$,}
\end{aligned}
\right.
\end{equation}
or, in a divergence-free formula
\begin{equation*}
\left \{
\begin{aligned}
& \text{find $\ubf_h \in \Zbf_h(\Omega_h)$ such that:} 
\\
& \mathcal K_h(\ubf_h, \, \vbf_h) = \mathcal{F}_h(\vbf_h) \quad \text{ $\forall \vbf_h \in \Zbf_h(\Omega_h)$.}\\
\end{aligned}
\right.
\end{equation*}
\begin{remark}
If the solution of problem \eqref{eq:oseen-variational-equation} satisfies $(\nablabf \ubf)\bb \in H^{\frac{5}{2}+\epsilon}(\Omega)$, then it holds
\[
\bigl[\!\!\bigl[\bigl( \nablabf
\ubf \bigr) \bb \cdot \mathbf t^E \bigr ]\!\! \bigr]_e = \bigl[\!\!\bigl[\mathcal B \ubf\bigr ]\!\! \bigr]_e
=
\bigl[\!\!\bigl[\nabla \mathcal B \ubf \bigr ]\!\! \bigr]_e
=
0
\]
for all internal edge $e \in \mathcal E_h^o$.
\end{remark}

\section{Theoretical analysis}

\subsection{Preliminary results}

Here we present some preliminary results that will be useful in the analysis of the method. We begin with the two following well-known results:
\begin{proposition}[Trace inequality]
Under assumptions \textbf{(A1)}, for all $v \in H^1(E)$, we have that
\[
\| v\|_{0,\partial E} \lesssim h_E^{-\frac{1}{2}} \| v_h\|_{0,E} + h_E^{\frac{1}{2}} | v_h|_{1,E} \quad \forall E \in \Omega_h\,  ,
\]
or equivalently
\[
\| v\|_{0,\partial E} \lesssim h_E^{-\frac{1}{2}} \| v_h\|^{\frac{1}{2}}_{0,E} \bigl( \| v_h\|_{0,E}+ h_E | v_h|_{1,E} \bigr)^{\frac{1}{2}}\quad \forall E \in \Omega_h  \, .
\]
\end{proposition}

\begin{proposition}[Poicarè-Friedrichs, \cite{brenner-sung:2018}] Under assumptions \textbf{(A1)}, for all $v \in H^1(E)$, we have that
\[
 \|v\|_{L^2(E)} \leq h_E^{-1} \left| \int_E v \, {\rm d} x\right| + h_E|v|_{H^1(E)} \quad \forall E \in \Omega_h \, ,
\]
and
\[
 \|v\|_{L^2(E)} \leq \left| \int_{\partial E} v \, {\rm d} s \right| + h_E|v|_{H^1(E)} \quad \forall E \in \Omega_h \, ,
\]
\end{proposition}
Furthemore, we need the following inverse estimates for the space $\Phi_h^k(\Omega_h)$.


\begin{proposition}[Inverse estimate $H^2(E)-H^1(E)$]\label{prp:inv21}
Under the mesh assumption \textbf{(A1)}, we have that
\[
| \phi_h|_{2,E} \lesssim h_E^{-1} |\phi_h|_{1,E} \qquad \forall \phi_h \in \Phi_h^k(\Omega_h) \, .
\]
\end{proposition}

\begin{proof}
Since we are in two dimensions, we have that the norm of the \textbf{curl} of a function is equivalent to the norm of the gradient. Hence we have
\[
| \phi_h |_{2,E} \approx |\textbf{curl} (\phi_h)|_{1,E} \, .
\]
Since the domain is simply connected, we have that $\textbf{curl} (\phi_h) = \vbf_h$ for some $\vbf_h \in \Vbf_h^k(\Omega_h)$.
This gives
\[
| \phi_h |_{2,E} \approx |\vbf_h|_{1,E} \, .
\]
Now, we use Theorem 2.1 in \cite{BMM:2023} and the fact that $\Div \vbf_h = 0$ to obtain
\[
|\vbf_h|^2_{1,E}
\leq
C_1 \bigl(h_E^{-2} \| \vbf_h\|^2_{0,E} + h_E^{-1} \| \vbf_h \|^2_{0,\partial E}\bigr) \, ,
\]
where the constant $C_1$ does not depend on the element size $h_E$.
Using a multiplicative trace inequality (whose constant is denoted with $C_t$) coupled with a Young inequality, we obtain
\[
\begin{aligned}
|\vbf_h|^2_{1,E}
&\leq
C_1 h_E^{-2} \| \vbf_h\|^2_{0,E} + C_1 h_E^{-1} \|  \vbf_h \|^2_{0,\partial E} \\
&\leq
C_1(1+C_t) h_E^{-2} \| \vbf_h\|^2_{0,E} 
+ 
C_1C_t h_E^{-1}\| \vbf_h\|_{0,E} | \vbf_h|_{1,E} \\
& \leq
C_1(1+C_t) h_E^{-2} \| \vbf_h\|^2_{0,E} 
+ 
\frac{C_1C_t}{\gamma} h_E^{-2}\| \vbf_h\|^2_{0,E} 
+
C_1C_t \gamma | \vbf_h|^2_{1,E}
\, ,
\end{aligned}
\]
and we take $\gamma$ sufficiently small such that $C_1 C_t \gamma < 1$ to obtain
\[
|\vbf_h|^2_{1,E} \lesssim \| \vbf_h\|^2_{0,E}\, .
\]
Gathering the previous inequalities, we conclude
\[
| \phi_h |_{2,E} \lesssim |\vbf_h|_{1,E}
\lesssim
h_E^{-1} \| \vbf_h\|_{0,E}
\lesssim
h_E^{-1} \| \textbf{curl} (\varphi_h)\|_{0,E}
\lesssim
h_E^{-1} |  \varphi_h|_{1,E} \, .
\]
\end{proof}

\begin{proposition}[Inverse estimate $H^1(E)-L^2(E)$] \label{prp:inv10}
Under the mesh assumption \textbf{(A1)}, we have that
\[
| \phi_h|_{1,E} \lesssim h_E^{-1} \|\phi_h\|_{0,E} \qquad \forall \phi_h \in \Phi_h^k(\Omega_h) \, .
\]
\end{proposition}

\begin{proof}
Using integration by parts, Proposition \ref{prp:inv21}, and a polynomial inverse estimate, we obtain
\[
\begin{aligned}
| \phi_h|^2_{1,E} 
&=
-\int_E \phi_h \, \Delta \phi_h {\rm d}s + \int_{\partial E} (\nabla \phi_h \cdot \mathbf n^E) \, \phi_h \, {\rm d}s \\
&\leq
\| \phi_h \|_{0,E} \, |\phi_h|_{2,E} + \| \phi_h \|_{0,\partial E} \, \|\nabla\phi_h\|_{0,\partial E} \\
&\leq
C_{v}\,h_E^{-1}\,\| \phi_h \|_{0,E} \,|\phi_h|_{1,E} 
+ 
C_{p}\, h_E^{-1} \, \| \phi_h \|_{0,\partial E}^2 \\
&\leq
C_{v}\,h_E^{-1}\,\| \phi_h \|_{0,E} \,|\phi_h|_{1,E} 
+ 
C_{p} \, C_t\, h_E^{-1} \, \| \phi_h \|_{0,E} \bigl( h_E^{-1} \, \| \phi_h \|_{0,E} + | \varphi_h|_{1,E} \bigr) \, .
\end{aligned}
\]
where $C_{v}$ is the constant that appears in Proposition \ref{prp:inv21}, $C_t$ is the constant in the trace inequality, and $C_p$ is the constant of the polynomial inverse inequality on the boundary. After defining
\[
\tilde C \coloneqq C_v + C_pC_t \, ,
\]
we have that
\[
|\varphi_h|^2_{1,E} \leq \tilde C \bigl (h_E^{-2} \| \varphi_h\|^2_{0,E} + h_E^{-1} \| \varphi_h\|_{0,E} | \varphi_h|_{1,E} \bigr) \, .
\]
We conclude using Young's inequality
\[
\| \varphi_h\|_{0,E} | \varphi_h|_{1,E} \leq 2\gamma\| \varphi_h\|_{0,E}^2 + \frac{1}{2\gamma}| \varphi_h|_{1,E}^2
\]
with $\gamma = 1/\tilde C$.
\end{proof}
\noindent Combining this two propositions, we easily obtain the following result.
\begin{corollary}[Inverse estimate $H^2(E)-L^2(E)$]
Under the mesh assumption \textbf{(A1)}, we have that
\[
| \phi_h|_{2,E} \lesssim h_E^{-2} \|\phi_h\|_{0,E} \qquad \forall \phi_h \in \Phi_h^k(\Omega_h) \, .
\]
\end{corollary}

\begin{remark}
The proof of Proposition \ref{prp:inv21} relies on the fact that the domain is two-dimensional. In this setting, the gradient of a function and the \textbf{curl} coincide as operators, up to a rotation. However, this equivalence does not hold in three dimensions, and therefore, this technique cannot be directly extended to the 3D case.
\end{remark}

We need to prove an inverse trace inequality for a class of virtual functions in the space $\Phi_h^k(\Omega_h)$. In particular, given an $E \in \Omega_h$, we would like to control the $L^2(E)$-norm of a virtual function whose internal DoFs are equal to zero with the $L^2(\partial E)$-norm.

\begin{proposition}
Under assumptions \textbf{(A1)}, if the internal degrees of freedom are equal to zero, we have that
\[
\| \varphi_h \|^2_{0,E}
\lesssim
h_E\|\varphi_h\|^2_{0,\partial E} 
+
h^3_E
\| \nabla \varphi_h \|^2_{0,\partial E} \, \quad \forall E \in \Omega_h, \, \forall \varphi_h \in \Phi_h^k(\Omega_h).
\]
\end{proposition}

\begin{proof}
We use two times the Poincarè-Friedrichs inequality
\[
\begin{aligned}
\| \varphi_h \|_{0,E}
&\lesssim
h_E | \varphi_h|_{1,E}
+ \left| \int_{\partial E} \varphi_h \, {\rm d}s \right| \\
&\lesssim
h^2_E | \varphi_h|_{2,E}
+ h_E \sum_{i=1}^2 \left| \int_{\partial E} \frac{\partial \varphi_h}{\partial x_i} \, {\rm d}s \right|
+ \left| \int_{\partial E} \varphi_h \, {\rm d}s \right| \, . \\
\end{aligned}
\]
Now, we observe that
\[
\left| \int_{\partial E} \varphi_h \, {\rm d}s \right| \leq h_E^{\frac{1}{2}} \|\varphi_h\|_{0,\partial E} \, ,
\quad \text{and}\quad
h_E \sum_{i=1}^2 \left| \int_{\partial E} \frac{\partial \varphi_h}{\partial x_i} \, {\rm d}s \right| \leq h_E^{\frac{3}{2}}\|\nabla\varphi_h\|_{0,\partial E} \, .
\]
On the $H^2(E)$-seminorm, we use again Theorem 2.1 in \cite{BMM:2023}. We obtain
\[
h_E^2 | \varphi_h|_{2,E} \lesssim h_E^2 | \textbf{curl} (\varphi_h)|_{1,E}
\lesssim 
h_E^{\frac{3}{2}} \| \textbf{curl} (\varphi_h)\|_{0,\partial E}  
\lesssim
h_E^{\frac{3}{2}} \| \nabla \varphi_h\|_{0,\partial E} \, .
\]
\end{proof}

We introduce an interpolation operator $\pi$ that maps a sufficiently piecewise smooth function to a virtual element function. This operator is commonly referred to as the Oswald interpolant or Averaging interpolant. The key idea is to average the DoFs shared by more than one element, while leaving the others unchanged.
The Oswald interpolant \cite{feI} has been applied to a large number of problems. In particular, we can mention \cite{burman:2004} and \cite{burman:2007} for an application to FEM. 
A VEM application of this interpolant can be found in \cite{BLT:2024} and \cite{LPT:2024}. 

\begin{proposition} \label{prp:oswald}
Under assumption \textbf{(A1)}, let $p \in \mathbb P_n(\Omega_h)$ a piecewise polynomial with $1 \leq n \leq k$. Let $\pi p$ the Oswald interpolant of $p$ into the space $\Phi^k_h(\Omega_h)$, then it holds
\[
\| ( I - \pi) p \|_{0,E}
\lesssim
h_E^{\frac{1}{2}} \sum_{e \in \mathcal{F}_E} \| [\![ p ]\!] \|_{0,e}
+
h_E^{\frac{3}{2}} \sum_{e \in \mathcal{F}_E} \| [\![ \nabla p ]\!] \|_{0,e}
\qquad \text{$\forall p \in \Pk_n(\Omega_h)$}\, ,
\]
where $\mathcal{F}_E \coloneqq \{ e \in \mathcal  E_h \, \text{ such that } \, e \cap \partial E \ne \emptyset \}$ is the set of the edges with at least one endpoint which is a vertex of $E$.
\end{proposition}

\begin{proof}
We fix an element $E\in\Omega_h$, and we consdier the difference
\[
d \coloneqq (I - \pi)p \, .
\]
We observe that since the internal degrees of freedom belongs to only one element, the moment of the curl of this function are equal to zero. We can then apply the previous proposition to obtain that
\[
\| d \|_{0,E}
\lesssim 
h_E^{\frac{1}{2}}\| d\|_{0,\partial E}
+
h^{\frac{3}{2}}\|\nabla d\|_{0,\partial E}
\]
We can conclude the proof by applying the same arguments that appears in \cite{BLT:2024} and \cite{LPT:2024}.
\end{proof}

\subsection{Coercivity}

We introduce the norm
\begin{equation*}
\| \vbf_h \|^2_{\mathcal K, E}
\coloneqq
\nu \| \nablabf \vbf_h \|^2_{0,E}
+
\sigma \| \vbf_h \|^2_{0,E}
+
J_h^E(\vbf_h, \vbf_h) \, ,
\end{equation*}
with global counterpart
\begin{equation*}
\| \vbf_h \|^2_{\mathcal K}
\coloneqq
\left(
\sum_{E \in \Omega_h} \| \vbf_h \|^2_{\mathcal K, E}
\right)^{\frac{1}{2}} \, .
\end{equation*}

The well-posedness of \eqref{eq:oseen-discrete-problem} is guaranteed by this coercivity result
\begin{proposition}
Under assumption \textbf{(A1)}, it holds that
\[
\| \vbf_h \|^2_{\mathcal K}
\lesssim
\mathcal K_h(\vbf_h, \vbf_h)
\quad
\forall \vbf_h \in \Vbf_h^k(\Omega_h) \, .
\]
\end{proposition}

\begin{proof}
We restrict our attention to an element $E \in \Omega_h$. First, we note that since the bilinear form $\cskewEh$ is skew-symmetric, it holds that
\[
\cskewEh(\vbf_h, \vbf_h) = 0 \, .
\]
Thanks to to property of the VEM-stabilization term, we have
\[
A^E_h(\vbf_h, \vbf_h) 
\gtrsim
\nu \, \| \nablabf \vbf_h \|_{0,E}^2
+
\sigma \, \| \vbf_h \|_{0,E}^2 \, .
\]
By definition of the bilinear form $\mathcal{K}_h^E(\cdot,\cdot)$, it is clear that
\[
\mathcal{K}_h^E(\vbf_h, \vbf_h)
\gtrsim
\| \vbf_h \|^2_{\mathcal K, E} \, .
\]
The proof is completed by summing over all the elements $E \in \Omega_h$.
\end{proof}

\subsection{Error analysis}
Before proving the error analysis, we have to introduce some interpolation operators different from the ones that appear in Lemma \ref{lm:oseen-interpolation} and Lemma \ref{lm:oseen-interpolation2}. In particular, we consider as interpolation operator in the space $\Phi_h^k(\Omega_h)$ the $L^2(\Omega)-$orthogonal projection.  

\begin{lemma}\label{lm:intf}
Under asssumption \textbf{(A1)}, for any $\phi \in \Phi(\Omega) \cap H^{s+2}(\Omega_h)$, let $\varphi_{\mathcal I} \in \Phi_h^k(\Omega_h)$ be the $L^2(\Omega)$-orthogonal projection of $\varphi$ into the space $\Phi_h^k(\Omega_h)$. It holds
\[
\sum_{E \in \Omega} \| \varphi - \varphi_{\mathcal{I}}\|^2_{0,E} 
+
\sum_{E \in \Omega} h_E^2| \varphi - \varphi_{\mathcal{I}}|^2_{1,E}
+
\sum_{E \in \Omega} h_E^4| \varphi - \varphi_{\mathcal{I}}|^2_{2,E}
\lesssim
\sum_{E \in \Omega_h} h_E^{2s+4} | \varphi|^2_{s+2,E} \, ,
\]
where $0 < s \leq k$.
\end{lemma}

\begin{proof}
Since $\varphi_{\mathcal{I}}$ is the best possible approximation with respect to the $L^2$-norm in $\Phi_h^k(\Omega_h)$, we have that
\begin{equation}\label{eq:inv-prel1}
\sum_{E \in \Omega_h} \| \varphi - \varphi_{\mathcal{I}}\|^2_{0,E} 
= 
\| \varphi - \varphi_{\mathcal{I}} \|^2_{0,\Omega}
\leq
\| \varphi - \tilde \varphi_{\mathcal{I}} \|^2_{0,\Omega}
=
\sum_{E \in \Omega_h} \| \varphi - \tilde \varphi_{\mathcal{I}}\|^2_{0,E} 
\leq
\sum_{E \in \Omega_h} h_E^{2s+4} |\varphi|^2_{s+2,E} \, ,
\end{equation}
where $\tilde \varphi_{\mathcal{I}}$ is the interpolant of Lemma \ref{lm:oseen-interpolation2}. 
To prove the results for the $H^1(\Omega_h)$-seminorm, we use triangular inequality, the inverse estimate in Proposition \ref{prp:inv10}, and \eqref{eq:inv-prel1}. We obtain
\[
\begin{aligned}
\sum_{E \in \Omega_h} h_E^2 | \varphi - \varphi_{\mathcal{I}}|^2_{1,E}
&\leq
\sum_{E \in \Omega_h} h_E^2 | \varphi -  \tilde\varphi_{\mathcal{I}}|^2_{1,E}
+
\sum_{E \in \Omega_h} h_E^2 | \tilde\varphi_{\mathcal{I}} - \varphi_{\mathcal{I}}|^2_{1,E} \\
&\lesssim
\sum_{E \in \Omega_h} h_E^{2s+4}|\varphi|^2_{s+2,E}
+
\sum_{E \in \Omega_h} \| \tilde\varphi_{\mathcal{I}} - \varphi_{\mathcal{I}}\|^2_{0,E} \\
&\leq
\sum_{E \in \Omega_h} h_E^{2s+4}|\varphi|^2_{s+2,E}
+
\sum_{E \in \Omega_h} \| \varphi - \tilde\varphi_{\mathcal{I}}\|^2_{0,E}
+
\sum_{E \in \Omega_h} \| \varphi - \varphi_{\mathcal{I}}\|^2_{0,E} \\
& \lesssim \sum_{E \in \Omega_h} h_E^{2s+4}|\varphi|^2_{s+2,E} \, .
\end{aligned}
\]
In order to prove the analogous result for the $H^2(\Omega_h)$-seminrom, we use a similar procedure.
\end{proof}

\begin{lemma} \label{lm:intf2}
Under assumption \textbf{(A1)}, for any $\vbf \in \Zbf(\Omega) \cap [H^{s+1}(\Omega_h)]^2$, let $\varphi$ such that $\textnormal{\textbf{curl}}(\varphi) = \vbf$ and let $\varphi_{\mathcal{I}}$ be the interpolant of $\varphi$ defined in Lemma \ref{lm:intf}.
We define $\vbf_{\mathcal{I}} \eqqcolon \textnormal{\textbf{curl}}(\varphi_\mathcal{I})$, it holds
\[
\sum_{E \in \Omega} h_E^2 \,\| \vbf - \vbf_{\mathcal{I}}\|^2_{0,E} 
+
\sum_{E \in \Omega} h_E^4 \,| \vbf - \vbf_{\mathcal{I}}|^2_{1,E}
\lesssim
\sum_{E \in \Omega_h} h_E^{2s+4} \,| \vbf|^2_{s+1,E} \, ,
\]
where $0 < s \leq k$.
\end{lemma}

\begin{proof}
It is sufficient to combine Lemma \ref{lm:stab} with  Lemma \ref{lm:intf}.
\end{proof}

\begin{remark}
We must impose the quasi-uniformity assumption due to the non-local nature of these interpolation operators. In fact, the interpolant defined in Lemma \ref{lm:intf} is the global orthogonal projection operator.
\end{remark}

To analyze the method, and obtain the optimal rate of convergence we make the following regularity assumptions: \\
\textbf{(A2) Regularity assumption.} It holds that
\[
\begin{aligned}
\ubf \in H^{k+1}(&\Omega_h) \, , \quad 
p \in H^k{(\Omega_h)} \, , \quad
\bb \in W^{3}_{\infty}(\Omega) \cap H^{k+1}(\Omega_h) \, ,\\
&\fbf \in H^{k+1}(\Omega_h) \, , \quad 
(\nabla u)\bb \in H^{\frac{5}{2} + \epsilon}(\Omega) \, ,
\end{aligned}
\]
where $\epsilon > 0$. Furthemore, we need to consider this extra assumption: \\
\textbf{(A1+) Mesh assumption.} Under the mesh assumptions \textbf{(A1)}, we suppose that
\begin{itemize}
    \item the mesh is quasi-uniform: it holds $h \leq \rho h_E$ for all $E \in \Omega_h$.
\end{itemize}
We also introduce the quantities
\[
\lambda_E \coloneqq \max\left\{\nu, \sigma \,  h_E^2 \right\}
\quad \text{and} \quad
\lambda \coloneqq \max_{E \in \Omega_h} \{ \lambda_E \} \, .
\]

\begin{proposition}\label{prp:abs-vel}
Let $\ubf$ be the solution of \eqref{eq:oseen-strong}, and let $\ubf_h$ be the solution of the discrete problem \eqref{eq:oseen-discrete-problem}. It holds that
\[
\| \ubf - \ubf_h \|_{\mathcal K} 
\lesssim
\| \errbf_\mathcal{I} \|_{\mathcal{K}}
+
\frac{\eta_f
+
\eta_A
+
\eta_c
+
\eta_J}{\|\errbf_h\|_{\mathcal{K}}}
\, ,
\]
where we have defined 
\[
\begin{aligned}
    \eta_f &\coloneqq | \mathcal{F}(\errbf_h) - \mathcal{F}_h(\errbf_h) | \, , \\
    \eta_A &\coloneqq | A(\ubf,\errbf_h) - A_h(\ubf_{\mathcal{I}},\errbf_h) | \, , \\
    \eta_c &\coloneqq | c^{\text{skew}}(\ubf,\errbf_h) - c^{\text{skew}}_h(\ubf_{\mathcal{I}},\errbf_h) |\, ,\\
    \eta_J &\coloneqq | J_h(\ubf_{\mathcal{I}},\errbf_h) | \, ,
\end{aligned}
\]
and
\[\errbf_\mathcal I \coloneqq \ubf -  \ubf_\mathcal{I}
\qquad
\errbf_h \coloneqq \ubf_h -  \ubf_\mathcal{I} \, .\]
\end{proposition}

\begin{proof}
Using triangular inequality, it holds that
\[
\| \ubf - \ubf_h \|_{\mathcal K}
\lesssim
\| \ubf - \ubf_\mathcal{I} \|_{\mathcal K}
+
\| \ubf_h - \ubf_\mathcal{I} \|_{\mathcal K} 
=
\| \errbf_\mathcal{I} \|_{\mathcal K}
+
\| \errbf_h \|_{\mathcal K}
\, .
\]
Exploiting the coercivity of the bilinear form, we have that
\[
\begin{aligned}
\| \errbf_h \|^2_{\mathcal K}
\lesssim
\mathcal{K}_h( \errbf_h, \errbf_h)
=
\mathcal{K}_h( \ubf_h - \ubf_\mathcal{I}, \errbf_h)
&=
\mathcal{F}_h(\errbf_h)
- 
\mathcal{K}_h(\ubf_\mathcal{I}, \errbf_h) \\
&=
\mathcal{F}_h(\errbf_h)
-
\mathcal{F}(\errbf_h)
+
\mathcal{K}(\ubf, \errbf_h)
- 
\mathcal{K}_h(\ubf_\mathcal{I}, \errbf_h) \, .
\end{aligned}
\]
We conclude by recalling the definitions of $\mathcal{K}(\cdot, \cdot)$ and $\mathcal{K}_h(\cdot, \cdot)$ \eqref{eq:k} and \eqref{eq:kh}.
\end{proof}
We now proceed to estimate the terms that appear in the previous proposition. We begin by estimating the interpolation error.
\begin{lemma}[Estimate of $\| \errbf_{\mathcal{I}} \|_{\mathcal K}$] \label{lm:ei}
Under assumptions \textbf{(A1+)} and \textbf{(A2)}, it holds that
\[
\| \errbf_{\mathcal{I}} \|^2_{\mathcal{K}} 
\lesssim \sum_{E \in \Omega_h}
\lambda_E \, h_E^{2k} \, | \ubf |^2_{k+1,E}
+
\sum_{E \in \Omega_h}
\delta \| \bb \|^2_{[W^2_{\infty}(\Omega_h)]^2} h_E^{2k+1} | \ubf |^2_{k+1,E} \, .
\]
\end{lemma}

\begin{proof}
Using Lemma \ref{lm:intf2}, we obtain that
\[
\sum_{E \in \Omega_h} \sigma \| \errbf_{\mathcal{I}} \|^2_{0,E}
\lesssim
\sum_{E \in \Omega_h} 
\sigma \, h_E^{2k+2} \, | \ubf |^2_{k+1,E}
\leq
\sum_{E \in \Omega_h} 
\lambda_E \, h_E^{2k} \, | \ubf |^2_{k+1,E} \, .
\]
Similarly, using the quasi-uniformity assumption, we obtain
\[
\sum_{E \in \Omega_h} \nu \| \nablabf \errbf_{\mathcal{I}} \|^2_{0,E}
\lesssim
\sum_{E \in \Omega_h} 
\lambda_E \, h_E^{2k} \, | \ubf |^2_{k+1,E} \, .
\]
We have to control the part of the norm related to the jump bilinear form. The first level of jumps is estimated by the regularity of $\bb$, a polynomial trace inequality, and Lemma \ref{lm:intf2}
\[
\begin{aligned}
\delta_1 \sum_{E \in \Omega_h} J_h^{E,1}(\errbf_\mathcal{I}, \errbf_\mathcal{I}) 
&\leq
\sum_{E \in \Omega_h} \delta_1 h_E^2 \bigl \| \bigl[\!\!\bigl[\bigl( \nablabf \P0 \errbf_{\mathcal I} \bigr) \bb \bigr ]\!\! \bigr] \bigr \|_{0,\partial E}^2 \\
&\lesssim
\sum_{E \in \Omega_h} \delta_1 h_E \, \| \bb \|^2_{[L^\infty(\Omega)]^2} \bigl \| \nablabf \P0 \errbf_{\mathcal I}\bigr \|_{0, E}^2 \\
& \lesssim
\sum_{E \in \Omega_h}\delta_1 \| \bb \|^2_{[L^\infty(\Omega)]^2} h_E^{2k+1} |u|^2_{k+1,E} \, .
\end{aligned}
\]
Similarly, we estimate the other two levels of jumps. It remains to control the VEM stabilization part of $J_h(\cdot,\cdot).$ 
Using Lemma \ref{lm:bh}, we have that
\[
\begin{aligned}
\sum_{E \in \Omega_h}\delta \, h_E \, S_h^E \bigl( (I-\P0) \errbf_{\mathcal{I}}, (I-\P0) \errbf_{\mathcal{I}}\bigr)
& \lesssim
\sum_{E \in \Omega_h}\delta \, h_E | (I-\P0) \errbf_{\mathcal{I}} |^2_{1,E} \\
& \lesssim
\sum_{E \in \Omega_h}\delta \, h_E^{2k+1} | \ubf |^2_{k+1,E} \, .
\end{aligned}
\]
\end{proof}

\begin{lemma}[Estimate of $\eta_f$]\label{lm:etaf} Under assumptions \textbf{(A1+)} and \textbf{(A2)}, it holds that
\[
\eta_f
\lesssim
\left( \sum_{E \in \Omega_h} \lambda_E^{-1} \,  h_E^{2k+4} \,  | \fbf |^2_{k+1,E} \right)^\frac{1}{2} \| \errbf_h \|_{\mathcal K} \, .
\]
\end{lemma}

\begin{proof}
Exploiting the orthogonality of $\P0$, and using Lemma \ref{lm:bh}, we have that
\[
\begin{aligned}
\eta_f
&=
\sum_{E \in \Omega_h}\bigl( ( I-\P0) \fbf, \errbf_h \bigr)_{0,E}
=
\sum_{E \in \Omega_h}\bigl( ( I-\P0) \fbf, ( I-\P0) \errbf_h \bigr)_{0,E} \\
&\leq 
\left(\sum_{E \in \Omega_h}\| ( I-\P0) \fbf \|^2_{0,E}\right)^\frac{1}{2}
\left(\sum_{E \in \Omega_h}\| ( I-\P0) \errbf_h \|^2_{0,E}\right)^\frac{1}{2} \\
&\lesssim
\left(\sum_{E \in \Omega_h}\min\left\{ \sigma^{-1} , \nu^{-1}h_E^2\right\}  h_E^{2k+2} |\fbf |^2_{k+1,E} \right)^\frac{1}{2} \left(\sum_{E \in \Omega_h}\| \errbf_h \|^2_{\mathcal K,E}\right)^\frac{1}{2} \\
&\lesssim
\left(\sum_{E \in \Omega_h}  \lambda_E^{-1} \,{h_E^{2k+4}} \, | \fbf |^2_{k+1,E}\right)^\frac12 \| \errbf_h \|_{\mathcal K}\, .
\end{aligned}
\]
\end{proof}

\begin{lemma}[Estimate of $\eta_A$]\label{lm:etaA} Under assumptions \textbf{(A1+)} and \textbf{(A2)}, it holds that
\[
\eta_A
\lesssim
\left( \sum_{E \in \Omega_h} \lambda_E \,   h_E^{2k} \,  | \ubf |^2_{k+1,E} \right)^\frac12 \| \errbf_h \|_{\mathcal K} \, .
\]
\end{lemma}

\begin{proof}
We have to estimate the consistency term and the VEM stabilization term
\[
\begin{aligned}
\eta_A 
&= 
\sum_{E \in \Omega_h} \nu \bigl( \nablabf \ubf, \nablabf \errbf_h \bigr)_{0,E}
-
\sum_{E \in \Omega_h} \nu \bigl( \PZ0P \nablabf \ubf_{\mathcal I}, \PZ0P \nablabf \errbf_h \bigr)_{0,E}
\\
& \qquad +
 \sum_{E \in \Omega_h} \sigma \bigl( \ubf, \errbf_h \bigr)_{0,E}
-
\sum_{E \in \Omega_h} \sigma \bigl( \P0 \ubf_{\mathcal I}, \P0 \errbf_h \bigr)_{0,E} \\
& \qquad + \sum_{E \in \Omega_h} \lambda_E S_h^E\bigl( (I-\P0) \ubf_{\mathcal I}, (I-\P0) \errbf_h\bigr)  \eqqcolon \eta_A^c + \eta_A^s \, .
\end{aligned}
\]
For the first one, similarly to the previous lemma, exploiting the orthogonality of the polynomial projections, the quasi-uniformity of the mesh, and using Lemma \ref{lm:bh} and Lemma \ref{lm:intf2}, it holds that
\[
\begin{aligned}
\eta_A^c
&= 
\sum_{E \in \Omega_h} \nu \bigl( \nablabf \ubf - \PZ0P \nablabf \ubf, \nablabf \errbf_h \bigr)_{0,E}
+
\sum_{E \in \Omega_h} \nu \bigl(\PZ0P \nablabf \ubf - \PZ0P \nablabf \ubf_{\mathcal I}, \PZ0P \nablabf \errbf_h \bigr)_{0,E} \\
& \qquad +
\sum_{E \in \Omega_h} \sigma \bigl( \ubf - \P0 \ubf, \errbf_h \bigr)_{0,E}
+
\sum_{E \in \Omega_h} \sigma \bigl( \P0 \ubf - \P0 \ubf_{\mathcal I}, \P0 \errbf_h \bigr)_{0,E} \\
& \lesssim
\left(\sum_{E \in \Omega_h} \nu h_E^{2k} | \ubf |^2_{k+1,E}\right)^\frac12 \| \errbf_h \|_{\mathcal{K}}
+
\left(\sum_{E \in \Omega_h} \sigma h_E^{2k+2} | \ubf |^2_{k+1,E}\right)^\frac12 \| \errbf_h \|_{\mathcal{K}} \\
&\lesssim
\left(\sum_{E \in \Omega_h} \lambda_E h_E^{2k} | \ubf |^2_{k+1,E}\right)^\frac{1}{2} \| \errbf_h \|_{\mathcal{K}}\, .
\end{aligned}
\]
Now, we have to estimate the VEM stabilization term. Recalling the definition of $\lambda_E$, and using Lemma \ref{lm:bh}, we obtain
\begin{equation} \label{eq:stab-estimate}
\begin{aligned}
\sum_{E \in \Omega_h}\lambda_E \, S_h^E\bigl((I-\P0)\ubf_{\mathcal{I}}, (I-\P0)\errbf_h\bigr)
&\lesssim
\sum_{E \in \Omega_h}\lambda_E \, |(I-\P0)\ubf_{\mathcal{I}} |_{1,E} |(I-\P0)\errbf_h|_{1,E} \\
& \lesssim
\Big(\sum_{E \in \Omega_h}\lambda_E^\frac{1}{2} \, \bigl( |(I-\P0)\errbf_{\mathcal{I}} |_{1,E} \\
&\qquad \qquad \qquad+ |(I-\P0)\ubf |_{1,E} \bigr)\Big)^{\frac{1}{2}} \! \|\errbf_h\|_{\mathcal{K}} \\
&\lesssim
\left(\sum_{E \in \Omega_h}\lambda_E \, h_E^{2k} \, | \ubf |^2_{k+1,E}\right)^\frac{1}{2} \, \| \errbf_h \|_{\mathcal{K},E} \, .
\end{aligned}
\end{equation}
The proof is completed by summing over all the elements $E \in \Omega_h$.
\end{proof}

\begin{lemma}[Estimate of $\eta_c$]\label{lm:etac}
Under assumptions \textbf{(A1+)} and \textbf{(A2)},
it holds that
\[
\eta_c 
\lesssim
\left(
\sum_{E \in \Omega_h}\left( \lambda_E^{-1}h_E^2 \, \| \bb \|^2_{[W^{k+1}_\infty(\Omega_h)]^2} + \frac{h_E}{\min\{\delta\}} \right) h_E^{2k}\| \ubf\|^2_{k+1,E}
\right)^\frac12 \| \errbf_h \|_\mathcal K\, .
\]
where we have omitted low order terms.
\end{lemma}

\begin{proof}
Fixed an element $E \in \Omega_h$, recalling that $c(\cdot,\cdot) = c^{\text{skew}}(\cdot,\cdot)$, we have to consider the following two local forms
\[
\eta_{c,A}^E 
=
\bigl( ( \nablabf \ubf) \bb, \errbf_h \bigr)_{0,E}
-
\bigl( (\mathbf{\Pi}^{0,E}_k 
 \nablabf \ubf_{\mathcal I}) \bb, \P0 \errbf_h \bigr)_{0,E} \, ,
\]
\[
\eta_{c,B}^E 
=
\bigl( \P0 \ubf_{\mathcal{I}}, (\mathbf{\Pi}^{0,E}_k 
 \nablabf \errbf_h) \bb \bigr)_{0,E} \, ,
-
\bigl( u, (\nablabf \errbf_h)\bb \bigr)_{0,E} \, .
\]
On the first one, adding and subtracting, integrating by parts, we have 
\[
\begin{aligned}
\eta_{c,A}^E 
&=
\bigl( ( \nablabf \ubf) \bb, \errbf_h \bigr)_{0,E}
-
\bigl( ( \nablabf \ubf_{\mathcal{I}}) \bb, \P0 \errbf_h \bigr)_{0,E}
+
\bigl( (I - \mathbf{\Pi}^{0,E}_k) 
 \nablabf \ubf_{\mathcal I} , (\P0 \errbf_h) \bb^T \bigr)_{0,E} \\
 &=
\bigl( ( \nablabf \errbf_{\mathcal{I}}) \bb, \errbf_h \bigr)_{0,E}
+
\bigl( ( \nablabf \ubf_{\mathcal{I}}) \bb, (I-\P0) \errbf_h \bigr)_{0,E}
+
\bigl( (I - \mathbf{\Pi}^{0,E}_k) \nablabf \ubf_{\mathcal I} , (\P0 \errbf_h) \bb^T \bigr)_{0,E} \\
& =
\int_{\partial E} (\bb \cdot \nn^E) (\errbf_{\mathcal I} \cdot \errbf_h) {\rm d}e
-
\bigl ( \errbf_{\mathcal I}, (\nablabf \errbf_h)\bb)_{0,E}
+
\bigl( ( \nablabf \ubf_{\mathcal{I}}) \bb, (I-\P0) \errbf_h \bigr)_{0,E} \\
& \qquad +
\bigl( (I - \mathbf{\Pi}^{0,E}_k) \nablabf \ubf_{\mathcal I} , (\P0 \errbf_h) \bb^T \bigr)_{0,E} \\
& =
\int_{\partial E} (\bb \cdot \nn^E) (\errbf_{\mathcal I} \cdot \errbf_h) {\rm d}e
-
\bigl ( \errbf_{\mathcal I}, (\nablabf \P0 \errbf_h)\bb)_{0,E} \\
& \qquad 
-
\bigl ( \errbf_{\mathcal I}, (\nablabf (I - {\Pi}^{0,E}_{k}) \errbf_h)\bb)_{0,E}
+
\bigl( (I - \P0)[( \nablabf \ubf_{\mathcal{I}}) \bb], (I-\P0) \errbf_h \bigr)_{0,E}  \\
& \qquad +
\bigl( (I - \mathbf{\Pi}^{0,E}_k) \nablabf \ubf_{\mathcal I} , (\P0 \errbf_h) \bb^T \bigr)_{0,E} \\
& =
\int_{\partial E} (\bb \cdot \nn^E) (\errbf_{\mathcal I} \cdot \errbf_h) {\rm d}e
-
\bigl ( \errbf_{\mathcal I}, (\nablabf \P0 \errbf_h)\bb)_{0,E} \\
& \qquad 
-
\bigl ( \errbf_{\mathcal I}, (\nablabf (I - {\Pi}^{0,E}_{k}) \errbf_h)\bb)_{0,E}
+
\bigl( (I - \P0)[( \nablabf \ubf_{\mathcal{I}}) \bb], (I-\P0) \errbf_h \bigr)_{0,E}  \\
& \qquad +
\bigl( (I - \mathbf{\Pi}^{0,E}_k) \nablabf \ubf_{\mathcal I} , (\Pi^{0,E}_0 \errbf_h) \bb^T \bigr)_{0,E}
+
\bigl( (I - \mathbf{\Pi}^{0,E}_k) \nablabf \ubf_{\mathcal I} , (\P0\errbf_h - \Pi^{0,E}_0 \errbf_h) \bb^T \bigr)_{0,E}
\\
&\eqqcolon
\eta_{c,1}^E
+
\eta_{c,2}^E
+
\eta_{c,3}^E
+
\eta_{c,4}^E
+
\eta_{c,5}^E 
+
\eta_{c,6}^E 
\, .
\end{aligned}
\]
Similarly, for the other term, we have that
\[
\begin{aligned}
\eta_{c,B}^E
&=
\bigl( \P0 \ubf_{\mathcal{I}} - \ubf, (\mathbf{\Pi}^{0,E}_k 
 \nablabf \errbf_h) \bb \bigr)_{0,E} 
-
\bigl( u, [(\mathbf{\Pi}^{0,E}_k - I) \nablabf \errbf_h]\bb \bigr)_{0,E} \\
&=
\bigl( \P0 \ubf_{\mathcal{I}} - \ubf, ( 
 \nablabf \P0 \errbf_h) \bb \bigr)_{0,E} 
-
\bigl( u\bb^T, (\mathbf{\Pi}^{0,E}_k \nablabf \errbf_h- I) \errbf_h\bb \bigr)_{0,E} \\
& \qquad
+
\bigl( \P0 \ubf_{\mathcal{I}} - \ubf, [\mathbf{\Pi}^{0,E}_{k} \nablabf \errbf_h - \nablabf\P0 \errbf_h 
] \bb \bigr)_{0,E} \\
&=
- \bigl( \errbf_I, (\nablabf \P0 \errbf_h) \bb \bigr)_{0,E}
+
\bigl( (\P0 - I) \ubf_{\mathcal{I}}, ( 
 \nablabf \P0 \errbf_h) \bb \bigr)_{0,E} \\
& \qquad-
\bigl( u \bb^T, (\mathbf{\Pi}^{0,E}_k  - I)\nablabf \errbf_h)_{0,E} 
+
\bigl( \P0 \ubf_{\mathcal{I}} - \ubf, [\mathbf{\Pi}^{0,E}_{k} \nablabf \errbf_h - \nablabf\P0 \errbf_h 
] \bb \bigr)_{0,E} \\
 &\coloneqq
 \eta_{c,2}^E 
 +
 \eta_{c,7}^E 
 +
 \eta_{c,8}^E 
 +
 \eta_{c,9}^E \, . 
\end{aligned}
\]
We proceed by estimating the terms $\eta_{c,i} \coloneqq \sum_{E \in \Omega_h} \eta_{c,i}^E$ for $i = 1, \ldots, 9$. \\
\textit{Estimate of $\eta_{c,1}$:} Since the functions $\errbf_\mathcal{I}$ and $\errbf_h$ are continuous across the elements boundaries and equal to zero on $\Gamma$, it holds that
\[
 \eta_{c,1}
=
\sum_{E \in \Omega_h} \int_{\partial E} (\bb \cdot \nn^E) (\errbf_{\mathcal I} \cdot \errbf_h) \, {\rm d}s =0 \, .
\]
\textit{Estimate of $\eta_{c,2}$:} This term is the most complicated term to estimate. The general idea is that since $\Div(\errbf_I) = 0$, it exists a potential $\varphi_I$ such that $\textbf{curl}(\varphi_I) = \errbf_I$. Exploiting the orthogonality of the interpolation operator defined in Lemma \ref{lm:intf}, we add the Oswald interpolant of $\text{curl}((\nablabf  \P0 \errbf_h)\bb)$.
\begin{equation}\label{eq:par 1}
\begin{aligned}
\eta_{c,2}^E 
&=
 - \bigl ( \errbf_{\mathcal I}, (\nablabf \P0 \errbf_h)\bb\bigr)_{0,E}
 = 
 - \bigl( \textbf{curl} (\varphi_{\mathcal I}),  (\nablabf \ \P0 \errbf_h)\bb\bigr)_{0,E} \\
 &= 
- \bigl( \varphi_{\mathcal I}, \text{curl}((\nablabf  \P0 \errbf_h)\bb)\bigr)_{0,E}
 + 
 \int_{\partial E} \varphi_{\mathcal{I}} \bigl[ ( \nablabf \P0 \errbf_h) \bb\bigr] \cdot \mathbf{t}^E  \, {\rm d} s \, .
\end{aligned}
\end{equation}
The boundary integral is estimated using trace inequality, Proposition \ref{prp:inv10}, quasi uniformity assumption, Lemma \ref{lm:intf}, and first level of CIP
\begin{equation}\label{eq:par 2}
\begin{aligned}
\sum_{E \in \Omega_h}
 \int_{\partial E} \varphi_{\mathcal{I}} \bigl[ (\nablabf \P0 \errbf_h) \bb\bigr] \cdot \mathbf{t}^E \, {\rm d} s
 &\leq
 \sum_{E \in \Omega_h}
 \| \varphi_{\mathcal I} \|_{0,\partial E}
 \| [\![(\nablabf \Pg \errbf_h) \bb \cdot \mathbf t^E ]\!]  \|_{0,\partial E} \\
 &\leq
 \left(\sum_{E \in \Omega_h}
 h_E^{-3}\| \varphi_{\mathcal I} \|^2_{0, E}\right)^\frac{1}{2} \!
 \left(\sum_{E \in \Omega_h}
 h_E^2\| [\![(\nablabf \Pg \errbf_h) \bb \cdot \mathbf t^E]\!] \|^2_{0,\partial E}\right)^\frac{1}{2} \\
 & \leq
 \left(\sum_{E \in \Omega_h}
 {\delta_1^{-1}} h_E^{2k + 1} | \ubf |^2_{k+1,E}\right)^\frac{1}{2} \| \errbf_h \|_{\mathcal{K}} \, .
\end{aligned}
\end{equation}
We introduce $\bb_h$ as the piecewise $\mathbb P_2(\Omega_h)$ approximation of $\bb$. Let $\mathbf p \coloneqq (\nablabf \Pg \errbf_h) \bb_h$, exploiting the orthogonality of the interpolation operator, we have that
\[
\begin{aligned}
\sum_{E \in \Omega_h} \bigl( \varphi_{\mathcal I}, \text{curl}((\nablabf  \P0 \errbf_h)\bb)\bigr)_{0,E} 
&= 
\sum_{E \in \Omega_h}  \bigl( \varphi_{\mathcal I}, \text{curl}((\nablabf  \P0 \errbf_h)\bb_h)\bigr)_{0,E} \\ 
& \qquad +  \sum_{E \in \Omega_h} \bigl( \varphi_{\mathcal I}, \text{curl}((\nablabf  \P0 \errbf_h)(\bb - \bb_h))\bigr)_{0,E} \\
&= 
\sum_{E \in \Omega_h}  \bigl( \varphi_{\mathcal I}, (I-\pi) \, \text{curl}((\nablabf  \P0 \errbf_h)\bb_h)\bigr)_{0,E} \\ 
& \qquad +  \sum_{E \in \Omega_h} \bigl( \varphi_{\mathcal I}, \text{curl}((\nablabf  \P0 \errbf_h)(\bb - \bb_h))\bigr)_{0,E} \, ,
\end{aligned}
\]
where $\pi$ is the Oswald interpolant from Proposition \ref{prp:oswald}, mapping into the space $\tilde{\Phi}_h^k(\Omega_h)$, which corresponds to the space $\Phi_h^k(\Omega_h)$ without imposing zero boundary conditions on the boundary. 

We have that
\begin{equation}\label{eq:par 3}
\sum_{E \in \Omega_h} \bigl( \varphi_{\mathcal I}, \text{curl}((\nablabf  \P0 \errbf_h)(\bb - \bb_h))\bigr)_{0,E}
\lesssim
\left(\sum_{E \in \Omega_h} \lambda_E^{-1} \,  | \bb |^2_{[W^3_\infty(\Omega)]^2} \, h_E^{2k+6} \, | \ubf|^2_{k+1,E}\right)^\frac{1}{2} \| \errbf_h \|_{\mathcal{K}} \, .
\end{equation}
On the other hand, using Cauchy-Schwarz inequality, we obtain 
\begin{equation}\label{eq:c est}
\begin{aligned}
\sum_{E \in \Omega_h}  \bigl( \varphi_{\mathcal I}, (I-\pi) \, \text{curl}((\nablabf  \P0 \errbf_h)\bb_h)\bigr)_{0,E}
&
\lesssim
\sum_{E \in \Omega_h}
\| \varphi_\mathcal I \|_{0,E}
\| (I-\pi) \, \text{curl}((\nablabf  \P0 \errbf_h)\bb_h) \|_{0,E}
\end{aligned}
\end{equation}
Thanks to Proposition \ref{prp:oswald}, we have that
\begin{equation} \label{eq:Jest}
\| (I-\pi) \, \text{curl}((\nablabf  \P0 \errbf_h)\bb_h) \|^2_{0,E}
\lesssim
h_E \| [\![\mathbf{p}]\!]\|^2_{0,\partial E} + h_E^{3} \| [\![\nablabf \mathbf{p}]\!]\|^2_{0,\partial E} \, ,
\end{equation}
where $\mathbf{p} \coloneqq \text{curl}((\nablabf  \P0 \errbf_h)\bb_h)$.
For the portions of $\partial E$ that are not on the boundary of the domain $\Omega$, adding and subtracting $\bb$, we note that
\begin{equation*}
\begin{aligned}
\sum_{E \in \Omega_h}\| [\![\mathbf p]\!]\|^2_{0,\partial E /\Gamma} 
&=
\sum_{E \in \Omega_h} \| [\![\text{curl}((\nablabf  \Pg \errbf_h)\bb_h)]\!]\|^2_{0,\partial E/\Gamma} \\
&=
\sum_{E \in \Omega_h}\| [\![\text{curl}((\nablabf  \Pg \errbf_h)\bb)]\!]\|^2_{0,\partial E/\Gamma}
+
\sum_{E \in \Omega_h}\| [\![\text{curl}((\nablabf  \Pg \errbf_h)(\bb-\bb_h))]\!]\|^2_{0,\partial E/\Gamma} \\
&\lesssim
\sum_{E \in \Omega_h} \delta_2^{-1}h_E^{-4} J_h^{E,2}(\errbf_h, \errbf_h)
+
\sum_{E \in \Omega_h} h_E^{4} | \bb |^2_{[W^2_\infty(\Omega)]^2} \| \text{curl}(\nablabf\Pg \errbf_h)\|^2_{0,\partial E/\Gamma} \\
&\lesssim
\sum_{E \in \Omega_h} \delta_2^{-1} h_E^{-4} J_h^{E,2}(\errbf_h, \errbf_h)
+
\sum_{E \in \Omega_h} \lambda_E^{-1}h_E | \bb |^2_{[W^2_\infty(\Omega)]^2} \|\errbf_h\|^2_{\mathcal K,E} \, .
\end{aligned}
\end{equation*}
Similarly, if $\partial E$ is contained in $\Gamma$, using polynomial trace inequality and inverse estimates, we have
\[
\begin{aligned}
\sum_{E \in \Omega_h}\| [\![\mathbf p]\!]\|^2_{0,\partial E \cap \Gamma} 
=
\sum_{E \in \Omega_h}\| \mathbf p\|^2_{0,\partial E \cap \Gamma} 
&\lesssim
\sum_{E \in \Omega_h} h_E^{-1} \|\text{curl}((\nablabf  \P0 \errbf_h)\bb_h)\|^2_{0,E} \\
&\lesssim
\sum_{E \in \Omega_h} \| \bb_h\|^2_{[L^\infty(E)]^2} \bigl(h_E^{-3}\| \P0 \errbf_h\|^2_{0,E} + h_E^{-1}| \P0 \errbf_h|^2_{1,E}\bigr) \\
&\lesssim
\sum_{E \in \Omega_h} \lambda_E^{-1} h_E^{-1} \| \bb_h\|^2_{[L^\infty(E)]^2}\| \errbf_h\|_{\mathcal K,E} \, .
\end{aligned}
\]
These two estimates combined give
\begin{equation} \label{eq:J2-est}
\sum_{E \in \Omega_h}\| [\![\mathbf p]\!]\|^2_{0,\partial E /\Gamma}  
\lesssim 
\sum_{E \in \Omega_h} \delta_2^{-1} h_E^{-4} J_h^{E,2}(\errbf_h, \errbf_h) 
+
\sum_{E \in \Omega_h} \lambda_E^{-1}h_E^{-1} \| \bb \|^2_{[W^2_\infty(\Omega)]^2} \|\errbf_h\|^2_{\mathcal K,E}
\end{equation}
Similarly, we obtain
\begin{equation}\label{eq:J3-est}
\sum_{E \in \Omega_h}\| [\![\nablabf \mathbf p]\!]\|^2_{0,\partial E}
\lesssim
\sum_{E \in \Omega_h} \delta_3^{-1} h_E^{-6} J_h^{E,3}(\errbf_h, \errbf_h)
+
\sum_{E \in \Omega_h} \lambda_E^{-1}h_E^{-1} \| \bb \|^2_{[W^2_\infty(\Omega)]^2} \|\errbf_h\|^2_{\mathcal K,E} \, .
\end{equation}
Combining \eqref{eq:J2-est} and \eqref{eq:J3-est} in \eqref{eq:Jest}, and then in \eqref{eq:c est}, using Lemma \ref{lm:stab}, we obtain
\[
\begin{aligned}
\sum_{E \in \Omega_h} \| (I-\pi) \, \text{curl}((\nablabf  \P0 \errbf_h)\bb_h) \|_{0,E}
&\lesssim
\left(\sum_{E \in \Omega} \lambda_E^{-1} | \bb |^2_{[W^3_\infty(\Omega)]^2} h_E^{2k+6} | \ubf|^2_{k+1,E}\right)^\frac{1}{2}\| \errbf_h \|_\mathcal K \\
& \quad
+ \left(\sum_{E \in \Omega_h} \max\{ \delta_2^{-1}, \delta_3^{-1}\}  h_E^{2k+1} | \ubf |^2_{k+1,E}\right)^\frac{1}{2} \| \errbf_h \|_{\mathcal K} \, .
\end{aligned}
\]
We conclude that
\[
\begin{aligned}
\eta_{c,2} &\lesssim 
\left(\sum_{E \in \Omega_h} \max\{\delta_i^{-1}\}_{i=1}^3 \, h_E^{2k+1} | \ubf|^2_{k+1,E}\right)^\frac12 \|\errbf_h\|_{\mathcal K} \\ 
& \qquad+
\left(\sum_{E \in \Omega_h} \lambda_E^{-1} | \bb |^2_{[W^2_\infty(\Omega)]^2} h_E^{2k+6} | \ubf|^2_{k+1,E}\right)^\frac12 \!\| \errbf_h \|_\mathcal K \, .
\end{aligned}
\]
\textit{Estimate of $\eta_{c,3}$:} On the third one, we use Cauchy-Schwarz inequality, Lemma \ref{lm:intf2}, and the definition of $\lambda_E$ to obtain
\[ 
\begin{aligned}
\eta_{c,3} 
&=
\sum_{E\in\Omega_h}
-\bigl ( \errbf_{\mathcal I}, ((I - \mathbf{\Pi}^{0,E}_{k-1}) \nablabf \errbf_h)\bb \bigr )_{0,E} \\
&\leq
\sum_{E\in\Omega_h} \| \bb \|_{[L^\infty(\Omega)]^2} \| \errbf_{\mathcal I} \|_{0,E}
\| (I - \mathbf{\Pi}^{0,E}_{k-1}) \nablabf \errbf_h) \|_{0,E} \\
& \lesssim
\left(\sum_{E \in \Omega_h} \lambda_E^{-1}
  \, \| \bb \|^2_{[L^\infty(\Omega)]^2} \, h_E^{2k+2}  \, | \ubf|^2_{k+1,E} \right)^\frac12   \| \errbf_h\|_{\mathcal K} \, .
\end{aligned}
\]
\textit{Estimate of $\eta_{c,4}$:} We begin by adding and subtracting $\bigl( (I - \P0)[( \nablabf \ubf) \bb], (I-\P0) \errbf_h \bigr)_{0,E}$ on each element $E \in \Omega_h$, we use Cauchy-Schwarz inequality, and Lemma \ref{lm:bh} and Lemma \ref{lm:intf2}.
\[
\begin{aligned}
 \eta_{c,4}
&= 
\sum_{E \in \Omega_h}
-
\bigl( (I - \P0)[( \nablabf \ubf_{\mathcal{I}}) \bb], (I-\P0) \errbf_h \bigr)_{0,E} \\
&\leq
\sum_{E \in \Omega_h}
\bigl( (I - \P0)[( \nablabf \ubf) \bb], (I-\P0) \errbf_h \bigr)_{0,E} \\
& \qquad -
\sum_{E \in \Omega_h}\bigl( (I - \P0)[( \nablabf \errbf_{\mathcal{I}}) \bb], (I-\P0) \errbf_h \bigr)_{0,E} \\
& \lesssim
\sum_{E \in \Omega_h} \left(\| (I - \P0)[( \nablabf \ubf) \bb] \|_{0,E} + \| (I - \P0)[( \nablabf \errbf_{\mathcal{I}}) \bb] \|_{0,E}  \right) \| (I-\P0) \errbf_h \|_{0,E}\\
&\lesssim
\sum_{E \in \Omega_h} \left( h_E^k |(\nablabf \ubf) \bb|_{k,E} + \| \bb \|_{[L^\infty(\Omega)]^2} \| \nablabf \errbf_{\mathcal{I}}\|_{0,E} \right) h_E \lambda_E^{-\frac{1}{2}} \, \|\errbf_h \|_{\mathcal{K},E} \\
&\lesssim
\left(\sum_{E \in \Omega_h}  \lambda_E^{-1} \,  \| \bb \|^2_{[W_\infty^k(\Omega)]^2} \, h_E^{2k+2} \|\ubf\|^2_{k+1,E} \right)^\frac12\|\errbf_h \|_{\mathcal{K}} \, .
\end{aligned}
\]
\textit{Estimate of $\eta_{c,5}$:} Using Cauchy-Schwarz inequality, Lemma \ref{lm:bh} and Lemma \ref{lm:intf2}, we get
\[
\begin{aligned}
\eta_{c,5}
&=
\sum_{E \in \Omega_h}
\bigl( (I - \mathbf{\Pi}^{0,E}_k) \nablabf \ubf_{\mathcal I} , (\Pi^{0,E}_0 \errbf_h) \bb^T \bigr)_{0,E} \\
&=
\sum_{E \in \Omega_h}
\bigl( (I - \mathbf{\Pi}^{0,E}_k) \nablabf \ubf_{\mathcal I} , (\Pi^{0,E}_0 \errbf_h)(I-\P0) \bb^T \bigr)_{0,E} \\
&\leq
\sum_{E \in \Omega_h}
\| (I - \mathbf{\Pi}^{0,E}_k) \nablabf \ubf_{\mathcal I} \|_{0,E} \| \Pi^{0,E}_0 \errbf_h\|_{0,E} \|(I-\P0) \bb^T \|_{[L^{\infty}(E)]^2} \\
& \lesssim
\left(\sum_{E \in \Omega_h}
\| \bb \|^2_{[W_\infty^{k+1}(E)]^2} \, h_E^{4k+2} \, |\ubf|^2_{k+1,E}\right)^\frac12\| \errbf_h \|_{0,\Omega} \, .
\end{aligned}
\]
Note that, thanks to the Poincarè inequality ($\errbf_h$ is zero on the boundary), we can estimate $\| \errbf_h \|_{0,\Omega}$ as
\[
\| \errbf_h \|_{0,\Omega}
\lesssim
\frac{\| \errbf_h \|_{\mathcal K}}{\sigma^{1/2}}
\qquad
\text{or}
\qquad
\| \errbf_h \|_{0,\Omega}
\lesssim
\| \nablabf \errbf_h \|_{0,\Omega}
\lesssim
\frac{\| \errbf_h \|_{\mathcal K}}{\nu^{1/2}} \, .
\]
\textit{Estimate of $\eta_{c,6}$:} Exploiting the orthogonality of the projection operators and using similar computations as in the previous cases, we obtain
\[
\begin{aligned}
\sum_{E \in \Omega_h} \eta^E_{c,6} 
&=
\sum_{E \in \Omega_h}\bigl( (I - \mathbf{\Pi}^{0,E}_k) \nablabf \ubf_{\mathcal I} , (\P0\errbf_h - \Pi^{0,E}_0 \errbf_h) \bb^T \bigr)_{0,E} \\
&=
\sum_{E \in \Omega_h}\bigl( (I - \mathbf{\Pi}^{0,E}_k) \nablabf \ubf_{\mathcal I} , (\P0\errbf_h - \Pi^{0,E}_0 \errbf_h) (I - \Pi^{0,E}_0)\bb^T \bigr)_{0,E} \\
&\leq
\sum_{E \in \Omega_h}\bigl \| (I - \mathbf{\Pi}^{0,E}_k) \nablabf \ubf_{\mathcal I} \|_{0,E} \, \|\P0\errbf_h - \Pi^{0,E}_0 \errbf_h\|_{0,E} \, \|(I - \Pi^{0,E}_0)\bb^T \|_{[L^\infty(E)]^2} \\
& \lesssim
\left(\sum_{E \in \Omega_h} \| \bb \|^2_{[W_\infty^1(\Omega)]^2} \, h_E^{2k+2} \, |\ubf|^2_{k+1,E} \right)^\frac{1}{2} 
\left(\sum_{E\in\Omega_h}\|(I - \Pi^{0,E}_0) \errbf_h\|^2_{0,E}\right)^\frac12 \, , 
\end{aligned}
\]
where the last term can be estimated as
\[
\|(I - \Pi^{0,E}_0) \errbf_h\|_{0,E}
\lesssim
\| \errbf_h \|_{0,E}
\lesssim
\frac{1}{\sigma^{\frac{1}{2}}}\| \errbf_h \|_{\mathcal K,E} \, ,
\]
or
\[
\|(I - \Pi^{0,E}_0) \errbf_h\|_{0,E}
\lesssim
h_E \| \nablabf \errbf_h\|_{0,E}
\lesssim
\frac{h_E}{\nu^{\frac{1}{2}}} \| \errbf_h\|_{\mathcal K,E} \, .
\]
Hence, we deduce
\[
\sum_{E \in \Omega_h} \eta_{c,6} ^E
\lesssim
\left(\sum_{E \in \Omega_h} \lambda_E^{-1}\| \bb \|^2_{[W_\infty^1(\Omega)]^2} \, h_E^{2k+4} \, |\ubf|^2_{k+1,E} \right)^\frac12 \| \errbf_h\|_{\mathcal K} \, .
\]
\textit{Estimate of $\eta_{c,7}$:} Similarly to the previous cases, using the orthogonaluty of the polynomial projections, Lemma \ref{lm:bh} and Lemma \ref{lm:intf2}, and the definition of $\lambda_E$, we obtain
\[
\begin{aligned}
\sum_{E \in \Omega_h} \eta_{c,7}^E 
&=
\sum_{E \in \Omega_h}\bigl( (\P0 - I) \ubf_I, (\nablabf \P0) \bb \bigr)_{0,E} \\
&=
\sum_{E \in \Omega_h}\bigl( (\P0 - I) \ubf_I, (\nablabf \P0) (I - \Pi^{0,E}_1)\bb \bigr)_{0,E} \\
&\lesssim
\left(\sum_{E \in \Omega_h}
\lambda_E^{-1} \, \| \bb \|^2_{[W^2_\infty(\Omega)]^2} \, h_E^{2k+6} \,| \ubf|^2_{k+1,E} \right)^\frac12\| \errbf_h\|_\mathcal K \, .
\end{aligned}
\]
\textit{Estimate of $\eta_{c,8}$:} Exploiting the orthogonality of the projection operators, and the same approximation/interpolation results, we obtain
\[
\begin{aligned}
\sum_{E \in \Omega_h} \eta_{c,8}^E
&=
\sum_{E \in \Omega_h}
\bigl( u \bb^T, (\mathbf{\Pi}^{0,E}_k  - I)\nablabf \errbf_h)_{0,E}\\
&=
\sum_{E \in \Omega_h}
\bigl( (I-\mathbf{\Pi}^{0,E}_k)(u \bb^T), (\mathbf{\Pi}^{0,E}_k  - I)\nablabf \errbf_h)_{0,E}\\
&\leq
\sum_{E \in \Omega_h}
\| (I-\mathbf{\Pi}^{0,E}_k)(u \bb^T)\|_{0,E} \| (\mathbf{\Pi}^{0,E}_k  - I)\nablabf \errbf_h)\|_{0,E}\\
&\lesssim
\left(\sum_{E \in \Omega_h}
\lambda_E^{-1} \,  h_E^{2k+2} \,\| \bb \|^2_{[W^{k+1}_\infty(\Omega)]^2} \, \| \ubf \|^2_{k+1,E} \right)^\frac{1}{2} \| \errbf_h\|_{\mathcal K} \,. 
\end{aligned}
\]
\textit{Estimate of $\eta_{c,9}$:} Similarly to the previous cases, we have
\[
\begin{aligned}
\sum_{E \in \Omega_h} \eta_{c,9}^E 
&=
\sum_{E \in \Omega_h}
\bigl( \P0 \ubf_{\mathcal{I}} - \ubf, [\mathbf{\Pi}^{0,E}_{k} \nablabf \errbf_h - \nablabf\P0 \errbf_h 
] \bb \bigr)_{0,E}\\
&\leq
\sum_{E \in \Omega_h}
\| \P0 \ubf_{\mathcal{I}} - \ubf \|_{0,E} \| [\mathbf{\Pi}^{0,E}_{k} \nablabf \errbf_h - \nablabf\P0 \errbf_h 
] \bb \|_{0,E} \\
&\leq
\sum_{E \in \Omega_h}
\| \P0 \ubf_{\mathcal{I}} - \ubf \|_{0,E} \, , \bigl(\| [\mathbf{\Pi}^{0,E}_{k} \nablabf \errbf_h - \nablabf \errbf_h ] \bb \|_{0,E} + \| [\nablabf\P0 \errbf_h - \nablabf \errbf ] \bb \|_{0,E}\bigr) \\
& \lesssim
\left(\sum_{E \in \Omega_h} \lambda_E^{-1} \| \bb \|^2_{[L^\infty(\Omega)]^2}  h_E^{2k+2} | \ubf |^2_{k+1,E} \right)^\frac{1}{2}\| \errbf_h\|_{\mathcal K} \, .
\end{aligned}
\]
\end{proof}

\begin{lemma}[Estimate of $\eta_J$] \label{lm:etaJ}Under assumptions \textbf{(A1+)} and \textbf{(A2)}, it holds that
\[
\eta_J \lesssim
\left( \sum_{E \in \Omega_h}\bigl(\delta \, \| \bb \|^2_{[W^{2}_\infty(E)]^2}h_E^{2k+1} + \delta \lambda_E^{-1} h_E^{2k+2}\bigr) \,  \, | \ubf|^2_{k+1,E} \right)^\frac12 \| \errbf_h \|_{\mathcal K} \, .
\]
\end{lemma}

\begin{proof}
Thanks to the property $ \bb \cdot \nablabf \ubf \in [H^{\frac{5}{2}+ \epsilon}(\Omega)]^2$, it holds that
\[
[\![ \bb \cdot \nablabf \ubf]\!]_e
=
[\![ \mathcal B \ubf]\!]_e
=
[\![ \nabla \mathcal B  \ubf]\!]_e
=
0 \, ,
\]
for every internal edge $e \in \mathcal E^o_h$. Hence, using Cauchy-Schwarz inequality and the definition of $\| \cdot \|_{\mathcal{K}}$, we have that
\[
\begin{aligned}
\sum_{E \in \Omega_h} \delta_1 J_h^{E,1} (\ubf_{\mathcal{I}} , \errbf_h)
& = 
\sum_{E \in \Omega_h} \frac{\delta_1}{2} \int_{\partial E} h_E^2 \bigl[\!\!\bigl[\bigl( \nablabf \Pg \ubf_{\mathcal{I}} \bigr) \bb \cdot \mathbf{t}^e \bigr ]\!\! \bigr] \bigl[\!\!\bigl[\bigl( \nablabf \Pg \errbf_h \bigr) \bb \cdot \mathbf{t}^e \bigr ]\!\! \bigr] \, {\rm d }s \\
& = 
\sum_{E \in \Omega_h} \frac{\delta_1}{2} \int_{\partial E} h_E^2 \bigl[\!\!\bigl[\bigl( \nablabf \ubf - \nablabf \Pg \ubf_{\mathcal{I}} \bigr) \bb \cdot \mathbf{t}^e \bigr ]\!\! \bigr] \bigl[\!\!\bigl[\bigl( \nablabf \Pg \errbf_h \bigr) \bb \cdot \mathbf{t}^e \bigr ]\!\! \bigr] \, {\rm d }s \\
&\lesssim 
\sum_{E \in \Omega_h} \delta_1 \, h_E^2 \bigl \| [\!\!\bigl[\bigl( \nablabf \ubf - \nablabf \Pg \ubf_{\mathcal{I}} \bigr) \bb  ]\!\! \bigr] \bigl \|_{0,\partial E} \, \bigl\| \bigl[\!\!\bigl[\bigl( \nablabf \Pg \errbf_h \bigr) \bb \cdot \mathbf{t}^e \bigr ]\!\! \bigr] \bigr\|_{0,\partial E} \\
&\lesssim 
\sum_{E \in \Omega_h} \delta_1^{\frac{1}{2}} \, h_E \, \bigl \| [\!\!\bigl[\bigl( \nablabf \ubf - \nablabf \Pg \ubf_{\mathcal{I}} \bigr) \bb  ]\!\! \bigr] \bigl \|_{0,\partial E} \, \| \errbf_h \|_{\mathcal{K},E} \\
&\lesssim 
\left(\sum_{E \in \Omega_h} \delta_1 \, \| \bb \|^2_{[L^\infty(\Omega)]^2} \, h_E^{2k + 1} \, | \ubf |^2_{k+1,E}\right)^\frac12 \| \errbf_h \|_{\mathcal{K}}
\, .
\end{aligned}
\]
Similarly, we estimate the other two jumps. Combining the three levels of jumps, we obtain
\[
\sum_{E \in \Omega_h} \sum_{i=1}^3 \delta_i \,  J_h^{E,i}(\ubf_\mathcal{I}, \errbf_h) \lesssim
\left(\sum_{E \in \Omega_h} \delta \, \| \bb \|^2_{[W^{2}_\infty(E)]^2} \, h_E^{2k+1} \, | \ubf|^2_{k+1,E}\right)^\frac{1}{2} \, \| \errbf_h \|_{\mathcal K} \, .
\]
The stabilization term is estimated as in \eqref{eq:stab-estimate}.
\end{proof}
We conclude this section by combining all of these lemmas in this theorem
\begin{theorem}[Convergence]
Under assumptions \textbf{(A1+)} and \textbf{(A2)}, let $\ubf$ the solution of \eqref{eq:oseen-variational-equation} and $\ubf_h$ be the solution of \eqref{eq:oseen-discrete-problem}, it holds that
\[
\| \ubf - \ubf_h\|^2_\mathcal K \lesssim \sum_{E \in \Omega_h} \Theta_E^2 \left( \lambda_E^{}h_E^{2k} + \delta \| \bb \|^2_{[W^{k+1}_\infty(\Omega_h)]^2} h_E^{2k+1} 
+ 
\delta \lambda_E^{-1} h_E^{2k+2}
+
\lambda_E^{-1} h_E^{2k+4}\right) \, ,
\]
where the constant $\Theta_E$ depends on $\|\ubf \|_{k+1,E}$ and $\|\fbf \|_{k+1,E}$, and we have omitted lower order term.
\end{theorem}

\subsection{Pressure error analysis}

Let $p$ and $p_h$ the pressure component of the solutions of problems \eqref{eq:oseen-variational-equation} and \eqref{eq:oseen-discrete-problem} respectively.
Let $p_\pi $ be such that
\[
p_\pi |_E \coloneqq \Pi^{0,E}_{k-1} \, p \qquad \forall E \in \Omega_h \, .
\]
We note that, since $ (p,1)_{\Omega} = 0$, we have that
$(p_\pi,1)_{0,\Omega} = 0$. In fact
\[
(p_\pi,1)_{0,\Omega} 
= 
\sum_{E \in \Omega_h} (p_\pi,1)_{0,E}
=
\sum_{E \in \Omega_h} (\Pi^{0,E}_{k-1} \,p_,1)_{0,E}
=
\sum_{E \in \Omega_h} (p_,1)_{0,E}
=
(p,1)_{0,\Omega} 
= 0 \, .
\]
Since we have shown that $p_\pi \in Q_h^k(\Omega_h)$, exploiting the inf-sup condition \eqref{eq:oseen-infsup}, we have that it exists $\wbf_p \in \Vbf_h^k(\Omega_h)$ such that
\begin{equation}\label{eq:infsup-2}
\frac{b(\wbf_p, p_h - p_\pi)}{\| \nablabf \wbf_p\|_{0,\Omega}} \geq \hat \kappa\| p_h - p_\pi\|_{0,\Omega} \, ,
\end{equation}
and without loss in generality, we can assume $\| \nablabf \wbf_p\|_{0,\Omega} = 1$ .

As done in Proposition \ref{prp:abs-vel}, we introduce the following error estimation for the pressure field

\begin{proposition}
Let $(\ubf,p)$ be the solution of the continuous problem \eqref{eq:oseen-variational-equation} and let $(\ubf_h, p_h)$ be the solution of the discrete problem \eqref{eq:oseen-discrete-problem}.
It holds that
\[
\| p - p_h \|_{0,\Omega} \leq \| p - p_\pi \|_{0,\Omega} + \xi_f + \xi_A + \xi_c + \xi_j \, , 
\]
where
\[
\begin{aligned}
     \xi_f &\coloneqq | \mathcal{F}(\wbf_p) - \mathcal{F}_h(\wbf_p) | \, , \\
    \xi_A &\coloneqq | A(\ubf,\wbf_p) - A_h(\ubf_{\mathcal{I}},\wbf_p) | \, , \\
    \xi_c &\coloneqq | c^{\text{skew}}(\ubf,\wbf_p) - c^{\text{skew}}_h(\ubf_{\mathcal{I}},\wbf_p) |\, ,\\
    \xi_J &\coloneqq | J_h(\ubf_{\mathcal{I}},\wbf_p) | \, .   
\end{aligned}
\]
\end{proposition}

\begin{proof}
Using triangular inequality, we have that
\[
\|p-p_h\|_{0,\Omega}
\leq
\|p - p_\pi\|_{0,\Omega}
+
\|p_\pi - p_h \|_{0,\Omega}
\]
Thanks to the definition of $p_\pi$ and the fact that $\Div(\wbf_p) \in \mathbb P_{k-1}(\Omega_h)$, it holds that
\[
b(\wbf_p, p - p_\pi) = \sum_{E \in \Omega_h} \int_E (p - p_\pi) \Div(\wbf_p) \, {\rm d}E = 0 \, .
\]
Exploiting the inf-sup condition \eqref{eq:infsup-2}, we have that
\[
\begin{aligned}
\hat \kappa \| p_\pi - p_h \|_{0,\Omega} 
&\leq
b(\wbf_p, p_h - p_\pi) 
=
b(\wbf_p, p_h - p) \\
&= 
\mathcal{F}_h(\wbf_p)
-
\mathcal F(\wbf_p)
+
\mathcal K(\ubf, \wbf_p)
-
\mathcal K_h(\ubf_h, \wbf_p) \, .
\end{aligned}
\]
The conclusion follows by recalling the definitions of $\mathcal K(\cdot, \cdot)$ and $\mathcal K_h(\cdot, \cdot)$.
\end{proof}

\begin{lemma}[Estimate of $\xi_f$]
Under assumptions \textbf{(A1+)} and \textbf{(A2)},
it holds that
\[
\xi_f \lesssim \left(\sum_{E \in \Omega_h}   h_E^{2k+4} | \fbf |^2_{k+1,E}\right)^\frac{1}{2} \, .
\]
\end{lemma}
\begin{proof}
The proof is very similar to Lemma \ref{lm:etaf}, we have that
\[
\begin{aligned}
\xi_f
&=
\sum_{E \in \Omega_h}\bigl( ( I-\P0) \fbf, \wbf_p \bigr)_{0,E}
=
\sum_{E \in \Omega_h}\bigl( ( I-\P0) \fbf, ( I-\P0) \wbf_p \bigr)_{0,E} \\
&\leq 
\sum_{E \in \Omega_h}\| ( I-\P0) \fbf \|_{0,E}
\| ( I-\P0) \wbf_p \|_{0,E} \\
&\leq
\left(\sum_{E \in \Omega_h}   h_E^{2k+4} | \fbf |^2_{k+1,E}\right)^\frac{1}{2} 
\left(\sum_{E\in\Omega_h}\| \nablabf \wbf_p \|_{0,E}\right)^\frac{1}{2} 
\lesssim
\left(\sum_{E \in \Omega_h}   h_E^{2k+4} | \fbf |^2_{k+1,E}\right)^\frac{1}{2} \, .
\end{aligned}
\]
\end{proof}

\begin{lemma}[Estimate of $\xi_A$] Under assumptions \textbf{(A1+)} and \textbf{(A2)}, it holds that
\[
\xi_A \lesssim
\left(\sum_{E \in \Omega_h}(\nu+\sigma + \lambda_E) \| \ubf - \ubf_h \|^2_{\mathcal K,E}\right)^\frac{1}{2} + \left(\sum_{E \in \Omega_h}
\lambda_E^2 \, h_E^{2k} \, | \ubf|_{k+1,E}\right)^\frac{1}{2}\, .
\]
\end{lemma}

\begin{proof}
Similarly to Lemma \ref{lm:etaA}, we have that
\[
\begin{aligned}
\xi_A 
&= 
\sum_{E \in \Omega_h} \nu \bigl( \nablabf \ubf, \nablabf \wbf_p \bigr)_{0,E}
-
\sum_{E \in \Omega_h} \nu \bigl( \PZ0P \nablabf \ubf_h, \PZ0P \nablabf \wbf_p \bigr)_{0,E}
\\
& \qquad +
 \sum_{E \in \Omega_h} \sigma \bigl( \ubf, \wbf_p \bigr)_{0,E}
-
\sum_{E \in \Omega_h} \sigma \bigl( \P0 \ubf_h, \P0 \wbf_p \bigr)_{0,E} \\
& \qquad + \sum_{E \in \Omega_h} \lambda_E S_h^E\bigl( (I-\P0) \ubf_h, (I-\P0) \wbf_p\bigr)  \eqqcolon \xi_A^c + \xi_A^s \, .
\end{aligned}
\]
Following the same procedure of Lemma \ref{lm:etaA}, it holds
\[
\begin{aligned}
\xi_A^c=&\sum_{E \in \Omega_h} \nu \bigl( \nablabf \ubf - \PZ0P \nablabf \ubf, \nablabf \wbf_h \bigr)_{0,E}
+
\sum_{E \in \Omega_h} \nu \bigl(\PZ0P \nablabf \ubf - \PZ0P \nablabf \ubf_h, \PZ0P \nablabf \wbf_h \bigr)_{0,E} \\
& \qquad +
\sum_{E \in \Omega_h} \sigma \bigl( \ubf - \P0 \ubf, \wbf_h \bigr)_{0,E}
+
\sum_{E \in \Omega_h} \sigma \bigl( \P0 \ubf - \P0 \ubf_h, \P0 \wbf_h \bigr)_{0,E} \\
& \lesssim
\left(\sum_{E \in \Omega_h} \nu h_E^{2k} | \ubf |^2_{k+1,E} + \nu\| \ubf - \ubf_h\|^2_{\mathcal K,E} \right)^\frac{1}{2}
+
\left(\sum_{E \in \Omega_h} \sigma  h_E^{2k+2} | \ubf |^2_{k+1,E} +  \sigma\| \ubf - \ubf_h  \|^2_{\mathcal{K,E}} \bigr)\right)^\frac{1}{2}  \,
 .
\end{aligned}
\]
For the VEM stabilization term, the estimate is analoguos to \eqref{eq:stab-estimate} 
\begin{equation}\label{eq:vem-stab2}
\begin{aligned}
\sum_{E \in \Omega_h}\lambda_E \, S_h^E\bigl((I-\P0) \ubf_h, (I-\P0)\wbf_p\bigr)
&\lesssim
\left(\sum_{E \in \Omega_h}\lambda_E \, \| \ubf - \ubf_h \|^2_{\mathcal K,E} + \lambda_E^2 h_E^{2k} |\ubf |^2_{k+1,E}\right)^\frac{1}{2}\, .
\end{aligned}
\end{equation}
\end{proof}

\begin{lemma}[Estimate of $\xi_c$]\label{lm:xic}
Under assumptions \textbf{(A1+)} and \textbf{(A2)},
it holds that
\[
\begin{aligned}
\xi_c &\lesssim 
\left[
\sum_{E \in \Omega_h} \left(\frac{\| \bb \|^2_{[L^\infty(\Omega)]^2}}{\max\{\nu, \sigma\}} 
+ \lambda_E^{-1}\| \bb \|^2_{[L^\infty(\Omega)]^2} h_E \right)\|u - u_h\|^2_{\mathcal{K,E}} 
\right]^\frac{1}{2} \\
& \qquad 
+ \left(\sum_{E \in \Omega_h} \| \bb \|^2_{[W^k_\infty(\Omega_h)]^2} h_E^{2k+2} \|u\|^2_{k+1,E}\right)^\frac{1}{2} \, .
\end{aligned}
\]
\end{lemma}
\begin{proof}
Since the bilinear forms are the same and the stabilization is not involved, it sufficient to mimic the proof of Lemma 5.13 in \cite{BDV:2021}.    
\end{proof}

\begin{lemma}[Estimate of $\xi_J$]
Under assumptions \textbf{(A1+)} and \textbf{(A2)},
it holds that
\[
\begin{aligned}
\xi_J 
&\lesssim
\left(\sum_{E \in \Omega_h} \delta^2 \| \bb \|^4_{[L^\infty(\Omega)]^2} h_E^{2k+2}| \ubf |^2_{k+1,E}\right)^\frac12 
+
\left(\sum_{E \in \Omega_h} \delta^2 \lambda_E^{-1} \| \bb \|^4_{[L^\infty(\Omega)]^2} h_E^{2}\| \ubf - \ubf_h \|_{\mathcal K,E}\right)^\frac12 \\
& \qquad
+ \left(\sum_{E \in \Omega_h} \delta h_E^2 \| \ubf - \ubf_h \|^2_{\mathcal K}\right)^\frac{1}{2} 
+
\left(\sum_{E \in \Omega_h}\delta^2 h_E^{2k+2} | \ubf |^2_{k+1,E}\right)^\frac12 \, .
\end{aligned}
\]
\end{lemma}

\begin{proof}
Similarly to Lemma \ref{lm:etaJ}, we obtain
\[
\begin{aligned}
\sum_{E \in \Omega_h} \delta_1 J_h^{E,1} (\ubf_{h} , \wbf_p)
& = 
\sum_{E \in \Omega_h} \frac{\delta_1}{2} \int_{\partial E} h_E^2 \bigl[\!\!\bigl[\bigl( \nablabf \Pg \ubf_{h} \bigr) \bb \cdot \mathbf t^e \bigr ]\!\! \bigr] \bigl[\!\!\bigl[\bigl( \nablabf \Pg \wbf_p \bigr) \bb \cdot \mathbf t^e \bigr ]\!\! \bigr] \, {\rm d }s \\
& = 
\sum_{E \in \Omega_h} \frac{\delta_1}{2} \int_{\partial E} h_E^2 \bigl[\!\!\bigl[\bigl( \nablabf \ubf - \nablabf \Pg \ubf_{h} \bigr) \bb \cdot \mathbf t^e \bigr ]\!\! \bigr] \bigl[\!\!\bigl[\bigl( \nablabf \Pg \wbf_p \bigr) \bb \cdot \mathbf t^e \bigr ]\!\! \bigr] \, {\rm d }s \\
& \lesssim
\left(\sum_{E \in \Omega_h} \delta^2\| \bb \|^4_{[L^\infty(\Omega)]^2} h_E^{3}  (\| \nablabf(I - \P0) \ubf  \|^2_{0,E} + \| \nablabf \P0 (\ubf -  \ubf_h)\|^2_{0,E} ) \!\right)^\frac{1}{2} \\
& \lesssim
\left(\sum_{E \in \Omega_h} \delta^2 \| \bb \|^4_{[L^\infty(\Omega)]^2} h_E^2  \bigl( h_E^{2k} | \ubf |^2_{k+1,E} + \lambda_E^{-1}\| \ubf - \ubf_h \|^2_{\mathcal K,E} \bigr)\right)^\frac{1}{2}\, .
\end{aligned}
\]
Similarly, we estimate the other levels of jumps, while the VEM stabilization part is estimated as in \eqref{eq:vem-stab2}.
\end{proof}

Combining the last four results, we obtain the following estimate.

\begin{theorem}[Convergence]
Under assumptions \textbf{(A1+)} and \textbf{(A2)}, let $(\ubf, p)$ the solution of \eqref{eq:oseen-variational-equation} and $(\ubf_h, p_h)$ be the solution of \eqref{eq:oseen-discrete-problem}, it holds that
\[
\begin{aligned}
\|p - p_h\|^2_{0, \Omega} 
&\lesssim 
\sum_{E \in \Omega_h} h_E^{2k} |p|_{k,E}^2 
+ \sum_{E \in \Omega_h} h_E^{2k+4} \|f\|^2_{k+1,E} \\
& \quad+ \sum_{E \in \Omega_h} 
\lambda_E^2 h_E^{2k} |u|^2_{k+1,E} + \sum_{E \in \mathcal{Q}_h} \| \bb \|^2_{[W^k_\infty(\Omega_h)]^2} h_E^{2k+2} \|u\|^2_{k+1,E}\\
& \quad + \sum_{E \in \Omega_h} 
\left(\sigma + \lambda_E + \| \bb \|^2_{[L^\infty(\Omega)]^2} h_E 
+ \frac{\| \bb \|^2_{[L^\infty(\Omega)]^2}}{\max\{\nu, \sigma\}} + \frac{\| \bb \|^2_{[L^\infty(\Omega)]^2} h_E^2}{\lambda_E} \right)
\|u - u_h\|^2_{\mathcal K}.
\end{aligned}
\]

\end{theorem}

\section{Numerical results}

We consider a family of problems in the unit square $\Omega = (0,1) \times (0,1)$ and the following error quantities for the velocity will be considered:

\begin{itemize}
 \item \textbf{$H^1-$seminorm error}
 $$
 e_{H^1} := \sqrt{\sum_{E\in\Omega_h}\left\|\nablabf(\ubf-\Pi_k^\nabla \ubf_h)\right\|^2_{0,E}}\,,
 $$
 
 \item \textbf{$L^2-$norm error}
 $$
 e_{L^2} := \sqrt{\sum_{E\in\Omega_h}\left\|\ubf-\Pg \ubf_h\right\|^2_{0,E}}\,.
 $$
\end{itemize}
For the pressures, we consider the quantity
\[
e_{p} := \sqrt{\sum_{E\in\Omega_h}\left\|p-p_h\right\|^2_{0,E}}\,.
\]
For the tests, we consider a family $\{ \Omega_h\}_h$ of Voronoi meshes.

\paragraph{Solution with a boundary layer}
We consider a solution with a boundary layer. 
We select as solution of the problem $\eqref{eq:oseen-strong}$ the velocity $\ubf = (u_1,u_2)^T$ defined as
\[
u_1(x,y) = 0 \, ,
\qquad
u_2(x,y) = x - \frac{\exp(\frac{x-1}{\nu}) - \exp(\frac{-1}{\nu})}{1 - \exp(\frac{-1}{\nu})} \, ,
\]
and the pressure 
\[
p(x,y) = \frac{1}{2} - y \, .
\]
We choose $\nu = 10^{-9}$, $\sigma = 1$ and the advective field
\[
\bb(x,y) \coloneqq
\left[
\begin{aligned}
    &y^2 \\
    &x^2
\end{aligned}
\right] \, ,
\] 
as in \cite{BBCG:2024}.
To demonstrate the advantages of the CIP term, we consider different values for $\delta_1, \delta_2$ and $\delta_3$. 
The problem is solved with $k=3$ on a mesh consisting of 256 unit squares.
A square mesh was chosen to better visualize the numerical solutions. The results are shown in Figure \ref{fig:oseen-delta}.
We observe that, without any CIP term, the numerical solution does not match the analytical solution. By adding the first level of $J_h(\cdot,\cdot)$, we obtain a numerical solution that includes a boundary layer, but the peak of the function is around 1.2. The peak is reduced in some regions of the domain with the addition of the second level of CIP. With the final level of CIP, we observe an improved solution.
\begin{figure}[!htb]

\begin{center}
\begin{tabular}{ccc}
\includegraphics[width=0.30\textwidth]{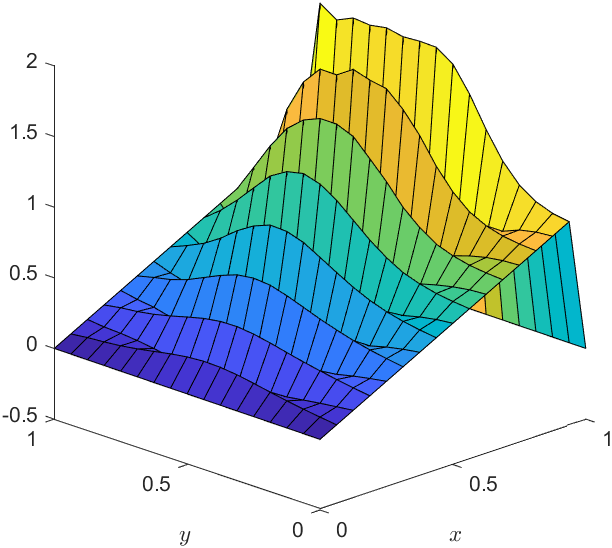} &\phantom{mm}&
\includegraphics[width=0.30\textwidth]{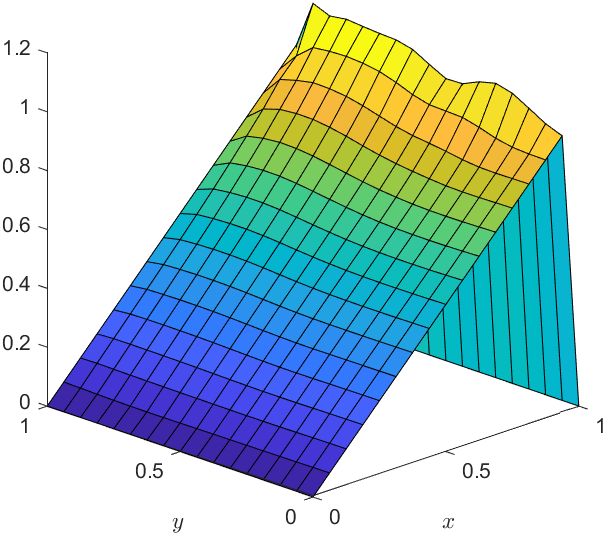} 
\end{tabular}

\begin{tabular}{ccc}
\includegraphics[width=0.30\textwidth]{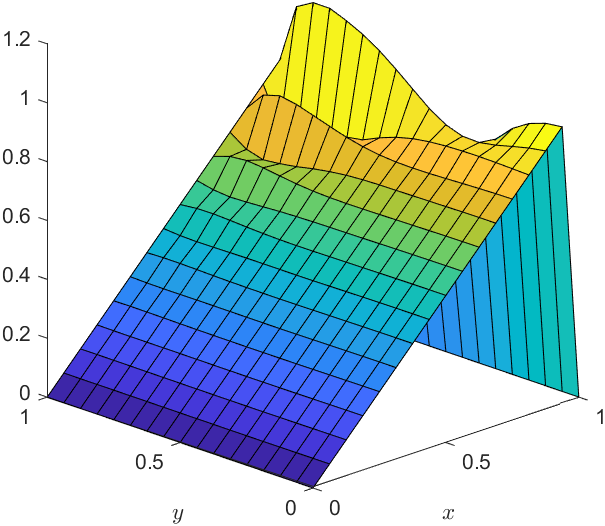} &\phantom{mm}&
\includegraphics[width=0.30\textwidth]{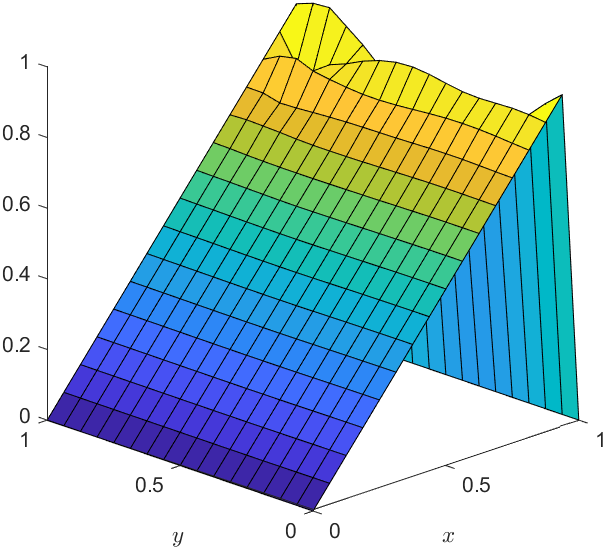} \\
\end{tabular}

\end{center}
\caption{Numerical solutions obtained with different choices of the triple $(\delta_1, \delta_2, \delta_3)$.Top-left $(\delta_1, \delta_2, \delta_3) = (0,0,0)$, top-right $(\delta_1, \delta_2, \delta_3) = (0.1,0,0)$, bottom-left $(\delta_1, \delta_2, \delta_3) = (0.1,0.01,0)$, bottom-right $(\delta_1, \delta_2, \delta_3) = (0.1,0.01,0.01)$ .}
\label{fig:oseen-delta}
\end{figure}

\paragraph{Convergence analysis.}
We consider as solution of \eqref{eq:oseen-strong} the couple $(\ubf, p)$ defined as
\[
\ubf(x,y) \coloneqq
\left[
\begin{aligned}
    -&\frac{1}{2} \sin(\pi x)^2 \cos(\pi y) \sin(\pi y) \\
    &\frac{1}{2} \sin(\pi y)^2 \cos(\pi x) \sin(\pi x)
\end{aligned}
\right] \, ,
\]
and
\[
p(x,y) \coloneqq \frac{1}{4}\bigl(\cos(4\pi x) - \sin(4\pi y)\bigr) \, .
\]
We choose as parameters of the method the values $\nu = 10^{-5}$, $\sigma = 1$ and the advective field defined as
\[
\bb(x,y) \coloneqq
\left[
\begin{aligned}
    &\sin(2\pi x) \sin(2 \pi y) \\
    &\cos(2\pi x) \cos(2 \pi y)
\end{aligned}
\right] \, .
\]
The CIP-parameters are set equal to $\delta_1 = 0.1$ and $\delta_2 = \delta_3 = 0.01$. 
We solve the equation using virtual element of order $ k = 2,3,4$. 
The results are depicted in Figure \ref{fig:oseen-velocities} and Figure \ref{fig:oseen-preassure}.
We can note that we have reached the optimal rate of convergence for the $H^1-$seminorm of the velocity and the $L^2-$norm of the pressure.
Actually, the method converges with a rate of $h^{k+1}$ in the $L^2$-norm of the velocity, which is a better result than the theoretical estimate for these types of problems that is $h^{k+\frac{1}{2}}$.
This might be due to the solution being very smooth.
\begin{figure}
\begin{center}
\begin{tabular}{ccccc}
\includegraphics[width=0.43\textwidth]{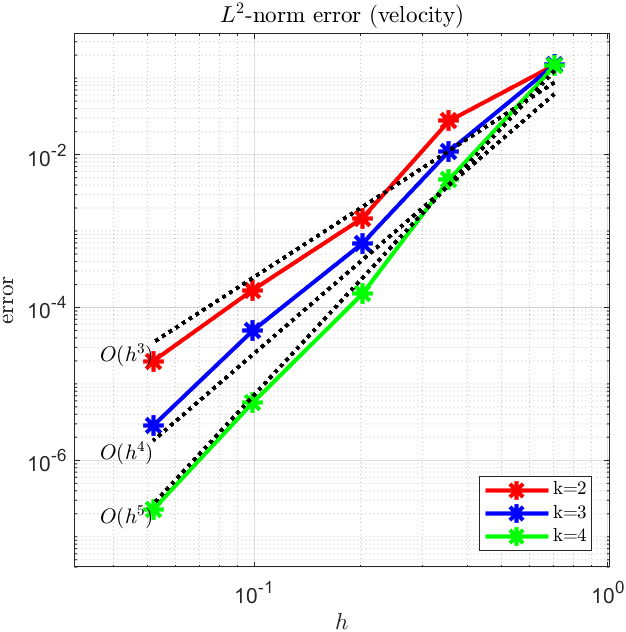}
&\phantom{mm}&
\includegraphics[width=0.43\textwidth]{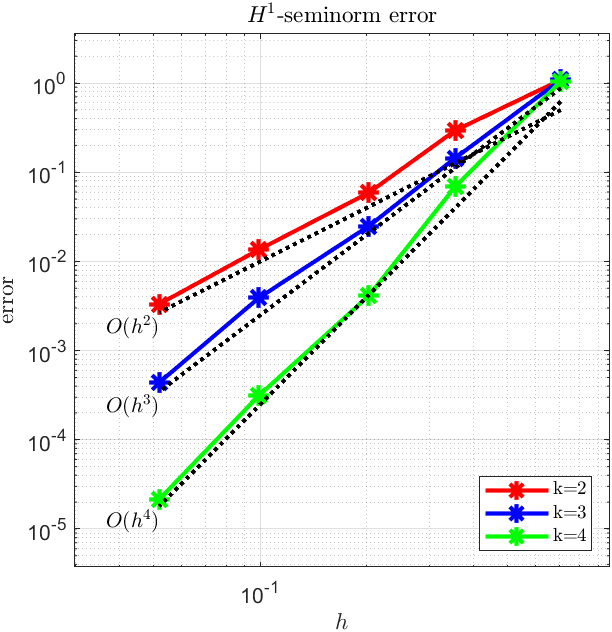} 
\end{tabular}
\end{center}
\caption{Convergences results for the velocity $\ubf$ in the $L^2$ norm (left column) and in the $H^1$-seminorm (right column). The red lines correspond to the case $k=2$, the blue lines to the case $k=3$, and the green lines to the case 
$k=4$.} 
\label{fig:oseen-velocities}
\end{figure}
\begin{figure}
\begin{center}
\begin{tabular}{ccc}
\includegraphics[width=0.43\textwidth]{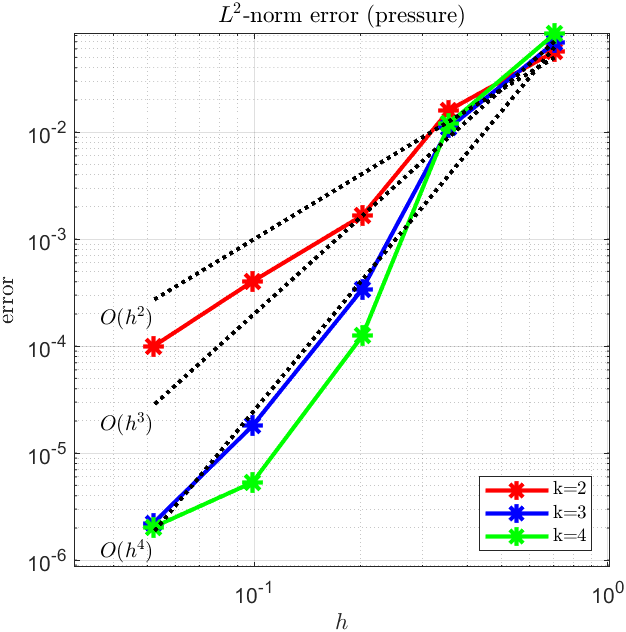}
\end{tabular}
\end{center}
\caption{Convergences for the pressure $p$ in the $L^2$-norm. The red lines correspond to the case $k=2$, the blue lines to the case $k=3$, and the green lines to the case 
$k=4$.} 
\label{fig:oseen-preassure}
\end{figure}

\paragraph{Pressure robustness.}
We want to verify if the method is pressure-robust in the VEM sense.
We consider the solution $(\ubf,p)$ of \eqref{eq:oseen-strong} defined as
\[
\ubf(x,y) = \underline 0 \, , \qquad p(x,y) = 3\cos(x) - 3\cos(y) \, .
\]
The parameters are set as in the first test case and select the order of the method $k=2,3,4$.
Since velocity $\ubf$ belongs to the discrete space, if the method is pressure robust, we expect to obtain an error of the order of the precision of the machine.
The VEM introduces in the error analysis of the velocity a little dependence on the pressure due to the approximation of the right-hand side.
This is a typical situation in VEM for this type of problems.
In Figure \ref{fig:oseen-pr}, we note that the errors are not of the order of the machine precision but they are quite low if compared to the previous cases.

\begin{figure}[!htb]

\begin{center}
\begin{tabular}{ccc}
\includegraphics[width=0.40\textwidth]{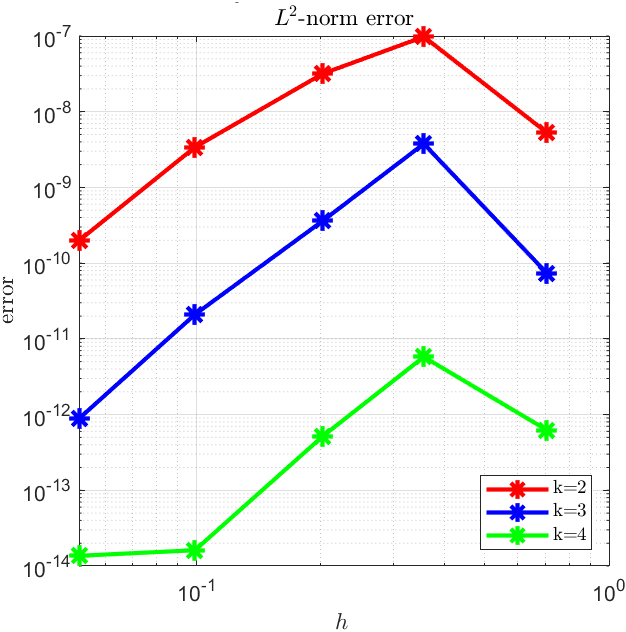} &\phantom{mm}&
\includegraphics[width=0.40\textwidth]{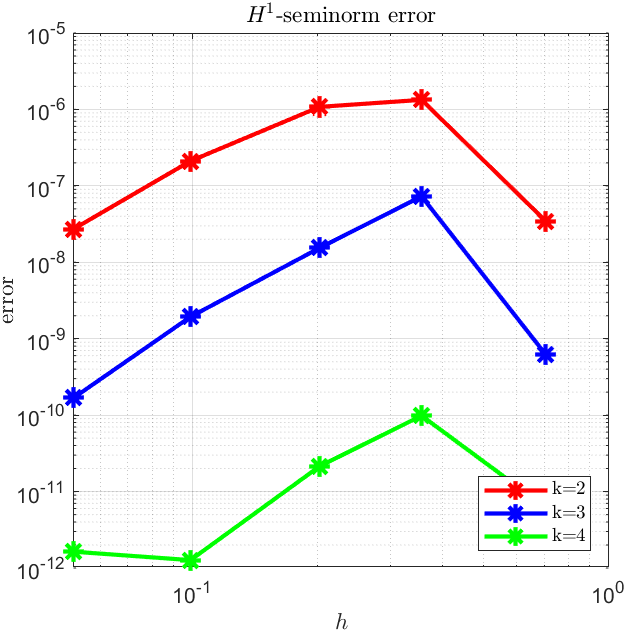} 
\end{tabular}

\end{center}
\caption{Result for the $L^2-$norm of the error (left column) and the $H^1-$seminorm of the error (right column). The red lines correspond to the case $k=2$, the blue lines to the case $k=3$, and the green lines to the case 
$k=4$. }
\label{fig:oseen-pr}
\end{figure}

\paragraph{Simulation of a fluid inside a channel with two pipes.}

In this test, we assume that a fluid, like water, is moving from left to right inside a channel, driven by an imposed flow velocity.
This channel contains two cylindrical obstacles, modeled as pipes positioned at different locations along the channel's length.
The domain is the same used in Section 3.3.
For simplicity, we assume that the flow velocity in the channel is uniform and directed along the horizontal axis, given by $\bb = (1,0)^T$.
We consider a scenario with low diffusion, setting the diffusion coefficient $\nu$ to $10^{-5}$, while neglecting the reaction term.
To ensure stability in the numerical solution, we apply the CIP stabilization method with parameters set to $(\delta_1, \delta_2, \delta_3) = (0.1, 0.01, 0.01)$.
No-slip boundary conditions are imposed along the top and bottom boundaries of the channel, as well as on the surfaces of the pipes.
On the left boundary, we prescribe an inflow condition with a parabolic velocity profile of the form
\[
-10\, (y-0.5)\,(y+0.5) \, .
\]
Finally, on the right boundary, we impose homogeneous Neumann boundary conditions to allow the fluid to exit the domain without further interaction.
The numerical solution in Figure \ref{fig:oseen-pipe} shows that the fluid retains its shape as it moves towards the first pipe.
Upon reaching the first obstacle, the fluid divides and flows both above and below the pipe, maintaining a symmetric shape.
This behavior is repeated at the second pipe, with the fluid continuing its trajectory in a similar pattern until it reaches the right boundary. We also tested the setup with square-based pipes to examine whether the corners would produce singularities, but we did not observe any significant differences.
We note that the following test cases do not fit within our theoretical analysis, as the domain contains two holes. Despite this, we attempt to assess the accuracy of our method. This test case is inspired by the one that appears in \cite{test}.
\begin{figure}[!htb]

\begin{center}
\includegraphics[width=0.97\textwidth]{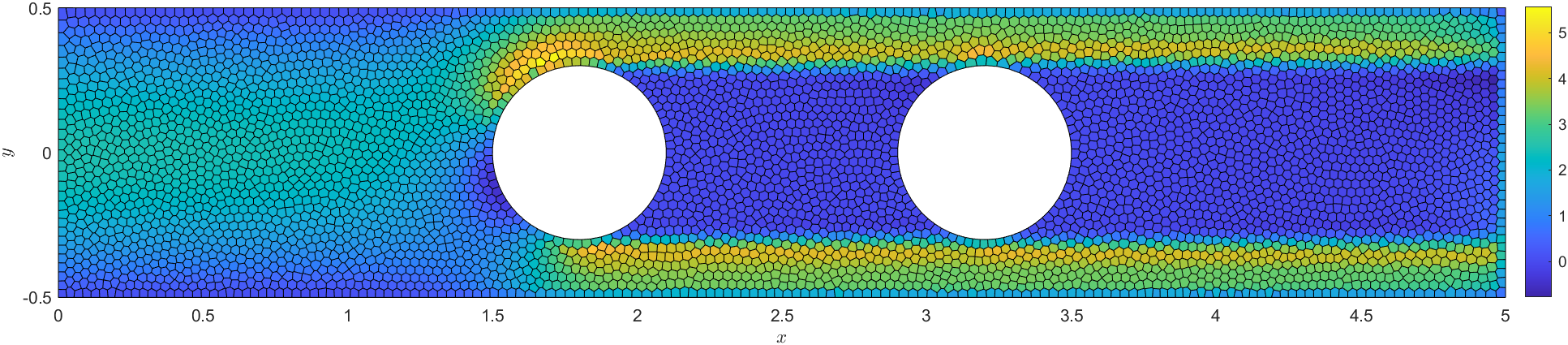} 
\end{center}
\caption{Numerical representation of a fluid that moves inside a channel with two pipes. The color represent the norm of the velocity.}
\label{fig:oseen-pipe}
\end{figure}

\section*{Acknowledgements}
\noindent
The author has been partially funded by the European Union (ERC, NEMESIS, project number 101115663).
Views and opinions expressed are however those of the author only and do not necessarily reflect those of the EU or the ERC Executive Agency.

\addcontentsline{toc}{section}{\refname}
\bibliographystyle{plain}
\bibliography{biblio}

\end{document}